\documentclass[a4paper,12pt]{article}

\usepackage{setspace} 
\usepackage{amsmath}
\usepackage{amsfonts}
\usepackage{amssymb}
\usepackage{ascmac}
\usepackage{amsthm}
\usepackage{amscd}
\usepackage{color}
\usepackage{enumerate}
\usepackage[dvips]{graphicx}
\usepackage{float}
\usepackage{wrapfig}

  \makeatletter
    
    \@addtoreset{equation}{section}
  \makeatother

\theoremstyle{definition}
\newtheorem{definition}{Definition}[section]
\theoremstyle{plain}
\newtheorem{theorem}[definition]{Theorem}

\newtheorem{lemma}[definition]{Lemma}
\newtheorem{proposition}[definition]{Proposition}
\theoremstyle{remark}
\newtheorem{remark}[definition]{Remark}

\begin{document}

\title{On a quarternification of complex Lie algebras}
\author{Tosiaki Kori 
\\Department of Mathematics\\
Graduate School 
of Science and Engineering\\
Waseda University,\\Tokyo 169-8555, Japan
\\email{ kori@waseda.jp}}
\date{}
\maketitle

\begin{abstract}
We give a definition of the {\it quarternification} of a complex Lie algebra.   By our definition \(\mathfrak{gl}(n,\mathbf{H})\),  \(\mathfrak{sl}(n,\mathbf{H})\), \(\mathfrak{so}^{\ast}(2n)\) and \(\mathfrak{sp}(n)\) are quarternifications of   \(\mathfrak{gl}(n,\mathbf{C})\),  \(\mathfrak{sl}(n,\mathbf{C})\), \(\mathfrak{so}(n,\mathbf{C})\) and \(\mathfrak{u}(n)\) respectively .   
We shall prove that a simple Lie algebra admits the  quarternification.      For the proof we follow the well known argument  to construct a simple Lie algebra from its root system.     
The root space decomposition of this quarternion Lie algebra will be given.    The root space of each fundamental root is complex 2-dimensional.   
 \end{abstract}


\medskip


\section{Introduction}

A {\it quarternion structure} on  a \(\mathbf{C}\)-module \(V\) is a conjugate linear map \(J:\,V\mapsto V\)  that  satisfies  the relation \(J^2\,=\,-\,I\,\), \cite{A}.   \(\,(V,J)\) is called a {\it quarternion  module} or a \(\mathbf{H}\)-{\it module}.    The  quarternion module \(V\) has  a decomposition: \(V=\mathbf{H}\otimes_{\mathbf{C}}V_o\simeq V_o+JV_o\) by a \(\mathbf{C}\)-submodule \(V_o\) and there is a conjugation \(\sigma\) on \(V\) defined by 
\(\sigma(\mathbf{u}_0+J\mathbf{v}_0)=\mathbf{u}_0-J\mathbf{v}_0\) for \(\mathbf{u}_0,\,\mathbf{v}_0\in V_o\).    
There exists also a complex conjugation automorphism  \(\tau\) defined by 
   \(\,
   \tau(\mathbf{u}_0+J\mathbf{v}_0)=\overline {\mathbf{u}_0}+J\overline{ \mathbf{v}_0}\,
\) for  \(\mathbf{u}_0,\,\mathbf{v}_0\in V_o\).     
Let \((V,J)\) be a quarternion module.   We call a \(\mathbf{R}\)-submodule \(W\) of \((V,J)\) a \(\sigma\)-{\it submodule} if \(W\) is invariant under the conjugations \(\sigma\) and \(\tau\).     As an example \(I=\{\,ia+jc\,;\,a\in \mathbf{R},\,c\in \mathbf{C}\}\) is a \(\sigma\)-submodule of \(\mathbf{H}\), but  is not a quarternion module.   We introduce  the concept of \(\sigma\)-{\it submodules} so that we can study Lie algebras like \(\mathfrak{sl}(n,\mathbf{H})\) or \(\mathfrak{so}^{\ast}(2n)\) both as a real Lie algebra and  as an object related to the quarternion structure.      
Let \(\mathfrak{g}\) be a \(\sigma\)-submodule of a quarternion module \((V,J)\).   We call \(\mathfrak{g}\)  a  {\it quarternion Lie algebra } if  \(\mathfrak{g}\) is endowed with a  real Lie algebra structure compatible with the conjugations: 
 \[ \sigma\,[X,Y]=\,[\sigma X,\sigma Y]\,,\quad \tau\,[X,Y]=[\tau X,\tau Y]\,.\]
The involution  \(\sigma\) has eigenvalues  \(\pm 1\).   
Let \(\mathfrak{g}=\mathfrak{g}^{+}+\mathfrak{g}^{-}\) denote the corresponding eigenspace decomposition.   \(\mathfrak{g}^{+}\) becomes a subalgebra of \(\mathfrak{g}\).
     Let  \(\mathfrak{g}_o\) be a complex Lie algebra.   We call a 
 quarternion Lie algebra \(\mathfrak{g}\)  {\it the quarternification of} \(\,\mathfrak{g}_o\) if 
 \(\mathfrak{g}_o\) is a subalgebra of 
 \(\mathfrak{g}^+\) and if 
\(\mathfrak{g}_o+\mathfrak{b}\) for a subspace \(\mathfrak{b}\) of \(\mathfrak{g}^-\)  generates \(\mathfrak{g}\) as a real Lie algebra.     
For example, the Lie algebra \(\mathfrak{so}^{\ast}(2n)\) of \(2n\times 2n\) skew symmetric complex matrices is the quarternification of the complex Lie algebra \(\mathfrak{so}(n,\mathbf{C})\) of  \(n\times n\) skew symmetric complex matrices.  
In fact we have a \(\mathbf{R}\)-algebra isomorphism 
\begin{eqnarray*}
\mathfrak{so}^{\ast}(2n)&=&\left\{ X\in\mathfrak{gl}(n,\mathbf{H})\,:\, \,^tX+X=0\,\right\}
\\[0.2cm]
&\simeq &
\left\{\left(\begin{array}{cc}A& -\overline B\\ B & \overline A\end{array}\right)\in \mathfrak{sl}(2n,\mathbf{C});\, 
\begin{array}{c} \,^tA+A=0 \\ \,^tB=B\end{array} \,\right\}.
\end{eqnarray*}
The latter is \(\mathbf{R}\)-isomorphic to \( \mathbf{H}\otimes_{\mathbf{C}}\mathfrak{so}(n,\mathbf{C})\) 
by the change of matrix representations from \(2n\times 2n\) complex matrix to \(n\times n\) quarternion matrix.   Hence \( \mathfrak{so}^{\ast}(2n)\simeq \mathbf{H}\otimes_{\mathbf{C}}\mathfrak{so}(n,\mathbf{C})=\mathfrak{so}(n,\mathbf{C})+J\mathfrak{so}(n,\mathbf{C})\).     By the same reasoning \(\mathfrak{sp}(n)\,=\,\left\{ X\in\mathfrak{gl}(n,\mathbf{H})\,:\, \,X^{\ast}+X=0\,\right\}\) is the quarternification of of the Lie algebra \(\mathfrak{u}(n)\).     As for the quarternification of complex Lie algebra \(\mathfrak{sl}(n,\mathbf{C})\), we can not have \( \mathbf{H}\otimes_{\mathbf{C}}\mathfrak{sl}(n,\mathbf{C})\) as the quarternification.   It is not a Lie algebra but it generates the real Lie algebra   
 \(\mathfrak{sl}(n,\mathbf{H})\).    \(\mathfrak{sl}(n,\mathbf{H})\) is a quarternifiction of \(\mathfrak{sl}(n,\mathbf{C}) \).     Note that  \(\mathfrak{sl}(n,\mathbf{H})\) is not a \(\mathbf{H}\)-module but a \(\sigma\)-submodule.  
  We shall prove that every simple  Lie algebra has the quarternification.   
 For the proof we apply the well known argument due to Chevalley, Harich-Chandra and Serre to construct the simple Lie algebra from its corresponding root system, \cite{S}.    Let \(\mathfrak{g}_o\) be a simple Lie algebra generated by the fundamental basis \(\{h_i, e_i, f_i\,;\,i=1,\cdots,l\,\}\,\) with the relations:
\begin{eqnarray*}
& [\,h_i,\,h_j\,]=0,\, \quad [h_i.e_j]=c_{ji}e_j\,,\quad
  [\,e_i \,,\,f_i\,] =\,\delta_{ij}\,h_i,\, .  \\[0.2cm]
&  (ad\,e_i)^{1-c_{ji}}(e_j)\,=\,0\,,\quad
(ad\,f_i)^{1-c_{ji}}(f_j)\,=\,0\,,\quad i\neq j\,.
\end{eqnarray*}
Where \(\left(c_{ij}\right)\) is the Cartan matrix of \(\mathfrak{g}_o\).   
  The quarternification \(\mathfrak{g}\) of \(\frak{g}_o\) is a quarternion Lie algebra generated by 
  \[\{\,h_i\,,\,e_i,\,f_i\,,\,Jh_i\,\,Je_i\,,\,Jf_i\,;\quad i=1,\cdots,l\,\}\]
  with the relations for \(\mathfrak{g}_o\,\) augmented by the following relations:
\begin{eqnarray*}
  [\,h_i,\,Jh_j\,]=0,\,& \quad  [\,h_i\,,\,Je_j\,]\,=\,c_{ji}\,Je_j\,, & \quad [\,h_i\,,\,Jf_j\,]\,=\,-c_{ji}\,Jf_j\,, 
 \nonumber\\[0.2cm]
 [\,Jh_i\,,\,e_j\,]\,=\,c_{ji}Je_j\,,&\quad [\,Jh_i\,,\,f_j\,]\,=\,- c_{ji}\,Jf_j\,, & \quad    [\,Jh_i,\,Jh_j\,]=0\,,
\,\nonumber\\[0.2cm]
   [\,e_i \,,\,Jf_i\,] =\,\delta_{ij}\,Jh_i,&\quad  [\,Je_i\,,\,f_i\,] =\,\delta_{ij}Jh_i, &\quad  [\,Je_i\,,\,Jf_i\,] =-\,\delta_{ij}\,h_i,  \,\nonumber\\[0.2cm]
   \,[\,Jh_i\,,\,Je_j\,]\,=\,-c_{ji}\,e_j\,,&\quad
[\,Jh_i\,,\,Jf_j\,]\,=\, c_{ji}f_j\,&\,.
\end{eqnarray*}
 \(\mathfrak{g}\) is a finite dimensional quarternion Lie algebra.     Let \(\mathfrak{h}_o\) be the Cartan subalgebra of the simple Lie algebra \(\mathfrak{g}_o\) and let \(\mathfrak{g}_o= \mathfrak{h}_o\oplus \sum_{\alpha\in \Phi}\,(\mathfrak{g}_o)_{\alpha}\) be the root space decomposition of \(\mathfrak{g}_o\) with the root space  
\(\,(\mathfrak{g}_o)_{\alpha}=\{\xi\in\mathfrak{g}_o;\,ad(h)\xi=\alpha(h)\xi, \quad\forall h\in \mathfrak{h}_o\}\).    
 The root system \(\Phi\) of  \(\mathfrak{g}_o\) may be taken as real linear forms: \(\Phi\subset Hom(\mathfrak{h}_r,\mathbf{R})\).   Where  \(\mathfrak{h}_r\) is the real form of \(\mathfrak{h}_o\).         \(\mathfrak{h}_r\) becomes an abelian subalgebra of the quarternion Lie algebra \(\mathfrak{g}\,\), and the triangular decomposition of \(\mathfrak{g}\) with respect to \(\mathfrak{h}_r\) is given by 
\begin{equation}
\mathfrak{g}=\mathfrak{k}\oplus \mathfrak{e}\oplus \mathfrak{f}\,, \quad \mbox{ with }\quad
 \mathfrak{e}\,=\,\sum_{\alpha\in\Phi_+}\mathfrak{g}_{\alpha}\,, \quad  \mathfrak{f}=\sum_{\alpha\in\Phi_-}\mathfrak{g}_{\alpha}\,,
 \end{equation}
where \(\,\mathfrak{g}_{\alpha}=\{\xi\in\mathfrak{g}\,;\,ad(h)\xi\,=\,\alpha(h)\xi, \quad\forall h\in \mathfrak{h}_r\}\).   
\(\mathfrak{k}\,\) is the ( real ) subalgebra generated by \(\mathbf{H}\otimes_{\mathbf{C} }\mathfrak{h}_o=\mathfrak{h}_o+J\mathfrak{h}_o\).    
\(\mathfrak{e}\) is generated by \(\{e_i,\,Je_i\,;\,i=1,\cdots,\,l\}\) and \(\mathfrak{f}\) is generated by \(\{f_i,\,Jf_i\,;\,i=1,\cdots,\,l\}\).   
For a non-zero root \(\alpha\in \Phi\,\), 
we have \(\mathfrak{g}_{\alpha}\,=\,\mathbf{H}\otimes_{\mathbf{C}} (\mathfrak{g}_o)_{\alpha}\,\).     \(\dim_{\mathbf{C}}\mathfrak{g}_{\alpha_i}=\dim_{\mathbf{C}}\mathfrak{g}_{-\alpha_i}=2\) for a simple root \(\alpha_i\), \(i=1,\,\cdots,\,l\,\).   
Each root space \(\mathfrak{g}_{\alpha}\) is a \(\mathbf{H}\)-module, so are the subalgebras \(\mathfrak{e}\) and \(\mathfrak{f}\),    
but \(\mathfrak{k}\) is a \(\sigma\)-submodule.     For example, for the quarternification \(\mathfrak{g}=\mathfrak{sl}(n,\mathbf{H})\) of \(\mathfrak{g}_o=\mathfrak{sl}(n,\mathbf{C})\),  we have 
 \(\mathfrak{e}=\sum_{i>j}\mathbf{H}\,E_{ij}\), \(\mathfrak{f}=\sum_{i<j}\mathbf{H}\,E_{ij}\) and \(\mathfrak{k}=\mathfrak{h}_o+\sqrt{-1}\mathbf{R}E_{11}+J(\mathfrak{h}_o+\mathbf{C}E_{11})\), where \(\mathfrak{h}_o=\sum_{i=1}^{n-1}\mathbf{C}(E_{ii}-E_{i+1\,i+1})\) is the Cartan subalgebra of \(\mathfrak{sl}(n,\mathbf{C})\).    

    In our forthcoming article \cite{K} we shall investigate current algebras of harmonic spinors on \(S^3\) that take values in a Lie algebra.   
 The spinors ae introduced by  the
 spinor representation of Clifford algebras: 
  \({\rm Clif }_4\,=\, {\rm End}_{\mathbf{R}}(\Delta)=\mathbf{H}(2)\,
\), 
 and the representation \(\Delta\) decomposes into irreducible representations  \(\Delta=\mathbf{H}\oplus\mathbf{H}\,\),  \(\Delta^{\pm}=\mathbf{H}\).   Hence even ( or odd ) spinors are identified with quarternion valued functions.  
    So one of the purpose of present article is to give a steady point of view for such a theory of spinor analysis as combined with Lie algebra theory via the {\it quarternification}.
 
There have been several trials to give the definition of {\it quarternion Lie algebras}.      R. Farnsteiner \cite{Fa} investigated a Lie algebra which is isomorphic to the central quotient of a quarternion division algebra and called it  {\it quarternion Lie algebra}.     
 D. Joyce \cite{J} and D. Widdow \cite{W} gave a definition of {\it quarternion Lie algebra} as an object of their \(\mathbf{AH}\)-module that satisfies the bracket conditions of Lie algebra.    
 \(\mathbf{AH}\)-module is a more restrictive concept than \(\mathbf{H}\)-module.   Their {\it quarternion Lie algebra} is fit to the smooth quarternion-valued vector fields on hypercomplex manifolds.   The above authors tried to give a definition of Lie algebras on the quarternion field.    While our Lie algebra is simply a real Lie algebra invariant under two conjugations of the quarternion field.

  \section{quarternion Lie algebras} 
  
  \subsection{quarternions $\mathbf{H}$}

Let  \(\mathbf{H}\) be the quarternion numbers.    A general quarternion is of the form \(\,x=x_1+x_2i+x_3j+x_4k\,\) with 
\(x_1,x_2,x_3,x_4\in \mathbf{R}\).      Every quarternion 
\(x\) has a unique expression   
\(x=z_1+jz_2\) with \(z_1,z_2\in\mathbf{C}\).    The quarternion multiplication will be from the right \(x\longrightarrow xy\) :
\begin{equation*}
xy=( z_1+jz_2\,)(  w_1+jw_2\,)=(z_1w_1-\overline z_2 w_2)+j(\overline z_1w_2+z_2w_1),
\end{equation*}
for \(x=z_1+jz_2\), \(y=w_1+jw_2\).    
Especially \(\mathbf{C}\) acts on \(\mathbf{H}\) from the right.   
 \(\mathbf{H}\) and \(\mathbf{C}^2\) are isomorphic as \(\mathbf{C}\)-vector spaces:
 \begin{equation}\label{hccoresp}
 \mathbf{H}\,\stackrel{\sim}{\longrightarrow}\,\mathbf{C}^2\,,\quad z_1+jz_2\longrightarrow
\left(\begin{array}{c}z_1\\z_2\end{array}\right)\,.
\end{equation}
The multiplication  of an element  \(g=a+jb\in \mathbf{H}\) from the left yields an endomorphism in \(\mathbf{H}\): \(\{x\longrightarrow gx\}\in End_{\mathbf{H}}(\mathbf{H})\).
Under the identification \(\mathbf{H}\simeq\mathbf{C}^2\)  the left quarternion multiplication  is expressed by a \(\mathbf{C}\)-linear map 
\begin{equation*}
\mathbf{C}^2\ni z=\left(\begin{array}{c}z_1\\z_2\end{array}\right)\,\longrightarrow gz=
\,\left(\begin{array}{cc}
a&-\overline b\\[0.2cm] b&\overline a\end{array}\right)\left(\begin{array}{c}z_1\\z_2\end{array}\right)\,\in \mathbf{C}^2\,.\end{equation*}
This establishes the \(\mathbf{R}\)-linear isomorphism
\begin{equation*}
\mathbf{H}\,\ni\, a+jb\,\stackrel{\simeq}{\longrightarrow}\, \left(\begin{array}{cc}
a&-\overline b\\[0.2cm] b&\overline a\end{array}\right)\,\in MJ(2,\mathbf{C})=\left\{\left(\begin{array}{cc}
a&-\overline b\\[0.2cm] b&\overline a\end{array}\right)\,:\quad a,b\in\mathbf{C}\right\},
\end{equation*}
and a \(\mathbf{R}\)-algebra isomorphism:
 \begin{equation}\label{qtomj}
 \mathbf{H}\simeq\,End_{\mathbf{H}}(\mathbf{H})\simeq MJ(2,\mathbf{C})\,.
 \end{equation}

\subsection{quarternion modules}

 A {\it quarternion structure} on  a  \(\mathbf{C}\)-vector space \(V\) is a conjugate linear map \(J:\,V\mapsto V\):  \(J(a\mathbf{u})=\overline{a}J\mathbf{u}\) for \(a\in \mathbf{C}\,, \mathbf{u}\in V\),  
   that  satisfies  the relation \(J^2\,=\,-\,I\,\), \cite{A}.    A left \(\mathbf{H}\)-module is a real vector space \(U\) with an action of \(\mathbf{H}\) on the left, \((x,\mathbf{v})
\longrightarrow\,x\mathbf{v}\), such that \(x(y\mathbf{v})=(xy)\mathbf{v}\) for all \(x,y\in \mathbf{H}\) and \(\mathbf{v}\in U\).    For example \(\mathbf{H}^n\) is an  \(\mathbf{H}\)-module.     A quarternion structure \(J\) on a \(\mathbf{C}\)-vector space \(V\) is equivalent to the \(\mathbf{H}\)-module structure on \(V\) viewed as a real vector space.   
    Let \(\sigma\) be a \(\mathbf{C}\)-linear involution on the quarternion module \((V,\,J)\) that  anti-commutes \(J\): \(\,J\sigma=\,-\sigma J\).   Let \(V_o\) be the eigensubspace of \(\sigma\) corresponding to the eigenvalue \(+1\).   
  Then \(JV_o=\{Ju; u\in V_o\}\) is the eigensubspace of \(\sigma\) corresponding to the eigenvalue \(-1\) and we have the direct sum decomposition of \(V\):
\begin{equation}\label{qtomj2}
V=V_o+JV_o\,\simeq\,\mathbf{H}\otimes_{\mathbf{C}} V_o\,.
   \end{equation} 
There is also a complex conjugation \(\tau\) on the \(\mathbf{C}\)-module \(V\); \(\tau^2=I\).   \(\tau\) restricts to the conjugate linear involution on \(V_o\,\);     \(\tau (z)=\overline z\), \(z\in V_o\). 
We have   \(\,\tau(z_1+Jz_2)=\overline z_1+J\overline z_2\,
\) for  \(z_1,\,z_2\in V_o\).      \(\sigma\) and \(\tau\) are commuting automorphisms of \(V\): \(\sigma\,\tau=\tau\,\sigma \).    If we let \(V_r\) denote the eigensubspace of \(\tau\) on \(V_o\) corresponding to the eigenvalue \(+1\), then we have 
 \[V_o=\mathbf{C}\otimes_{\mathbf{R}}V_r=V_r+\sqrt{-1}V_r\,,\qquad V=\mathbf{H}\otimes_{\mathbf{R}}V_r\,,\]
 that recovers the quarternionic structure on \(V\).   

\begin{definition}\label{submodule}~~
Let  \(V=V_o+JV_o\,\) be a \(\mathbf{H}\)-module.   Let \(W\) be a \(\mathbf{R}\)-submodule of \(V\).
\begin{enumerate}
 \item 
 \(W\) is called a \(J\)-{\it submodule} of \(V\) if \(W\) is invariant under \(J\), or equivalently, \(W\) has the direct sum decomposition 
 \(W=U_o+JU_o\) for a \(\mathbf{R}\)-vector subspace \(U_o\) of \(V_o\).   
 \item
 \(W\) is called a \(\sigma\)-{\it submodule} of \(V\) if \(W\) is invariant under the conjugate automorphisms \(\tau\) and  \(\sigma\), or equivalently, \(W\) respects the \(\mathbf{Z}_2\)-gradation; \(W=W\cap V_o+W\cap JV_o\).  
\end{enumerate}
 \end{definition}
A  \(J\)-submodule is a \(\sigma\)-submodule.    A \(\sigma\)-submodule \(W=W\cap V_o+W\cap JV_o\,\) is not necessarily a \(\mathbf{C}\)-module.

 In the following,  \(\mathfrak{gl}(n,\mathbf{H})\) denotes the algebra of \(n\times n\)-matrices with entries in  the quarternion numbers \(\mathbf{H}\).   
 
 {\bf Examples}
  \begin{enumerate}
   \item
  Every \(\mathbf{H}\)-submodule of a \(\mathbf{H}\)-module \(V\) is a \(J\)-submodule of \(V\).   \\
 \(\mathfrak{gl}(n,\mathbf{H})\), is a \(\mathbf{H}\)-module.    
 \item
 \(\mathbf{R}^n+J\mathbf{R}^n\) is a \(J\)-submodule of \(\mathbf{H}^n\), but is not a \(\mathbf{H}\)-submodule.   In general, for a \(\mathbf{H}\)-module \(V\),   \(V_r+JV_r\,\) is a \(J\)-submodule  but is not a \(\mathbf{H}\)-submodule.  
 \item
  \(I=\{ia+jc\,;\,a\in \mathbf{R},\,c\in \mathbf{C}\}\) is a \(\sigma\)-submodule of \(\mathbf{H}\) which is not a \(J\)-submodule.   
\item
\(K=\{(1+j)a\,;\,a\in \mathbf{R}\}\) is not a \(\sigma\)-submodule of \(\mathbf{H}\) but a real subspace.
 \item
 Let
 \[\mathfrak{sl}(n,\mathbf{H})\,=\,\{X\in \mathfrak{gl}(n,\mathbf{H});\quad {\rm Re.}Tr\,X=0\,\}.\]
  \(\mathfrak{sl}(n,\mathbf{H})\) is a \(\sigma\)-submodule of  \(\mathfrak{gl}(n,\mathbf{H})\) which is not a  \(J\)-submodule.   
\item 
We shall deal with Lie algebras \(\mathfrak{sp}(n)\) and \(\mathfrak{so}^{\ast}(2n)\) in the next section.   They are \(\sigma\)-submodules of \(\mathfrak{gl}(n,\mathbf{H})\).
  \end{enumerate}

    Let  \(V=\mathbf{H}\otimes_{\mathbf{C}} V_o\) and  \(W=\mathbf{H}\otimes_{\mathbf{C}} W_0\) be \(\mathbf{H}\)-modules.      A homomorphism of \(\mathbf{H}\)-module is by definition a 
\(\mathbf{R}\)-linear map \(T:\,V\longrightarrow W\) such that \(T(x\mathbf{v}\,)=\,xT(\mathbf{v})\,\) for \(\forall  x\in\mathbf{H}\) and \(\forall\mathbf{v}\in V\).    For \(\mathbf{H}\)-modules \(\,V\) and \(W\), We denote by \(Hom_{\mathbf{H}}(V,W)\)   the homomorphisms of \(\mathbf{H}\)-modules.    
\(Hom_{\mathbf{H}}(V,W)\) is a \(\mathbf{H}\)-module if the left action is defined by 
\[(cT)(\mathbf{v})\,=\,T(c\mathbf{v})\,,\quad \mbox{ for }\forall c \in \mathbf{H}\,,\,\forall\mathbf{v}\in V.\]
We denote \(End_{\mathbf{H}}(V)=Hom_{\mathbf{H}}(V,V) \).     

The decomposition \(V=V_o+JV_o\) yields the following 
\(\mathbf{Z}_2\)-gradation of the \(\mathbf{C}\)-module \(End_{\mathbf{C}}(V)\):
\begin{eqnarray}\label{2-grad}
End_{\mathbf{C}}(V)&=&\,End^0_{\mathbf{C}}(V)\oplus End^1_{\mathbf{C}}(V)\\[0.2cm]
End^0_{\mathbf{C}}(V)&=& \, Hom_{\mathbf{C}}(V_o,V_o)\oplus  Hom_{\mathbf{C}}(JV_o,JV_o)\,,\nonumber\\[0.2cm]
End^1_{\mathbf{C}}(V)&=& \, Hom_{\mathbf{C}}(JV_o,V_o)\oplus  Hom_{\mathbf{C}}(V_o,JV_o)\,.\nonumber
\end{eqnarray}
Given a basis of \(V_o\,\), any \(F\in End_{\mathbf{C}}(V)\) has the matrix representation.
\begin{equation}
F=\left (\begin{array}{cc}A&C\\[0.2cm]B&D\end{array}\right)\,:\,
\begin{array}{ccc}
V_o&\,&V_o\\[0.2cm]
\oplus &\longrightarrow&\oplus \\[0.2cm]
JV_o&\,& JV_o\,.\end{array}.
\end{equation}

Put 
\begin{equation}
End^J_{\mathbf{C}}(V)=\,\left\{F\in End_{\mathbf{C}}(V);\quad JF=\overline FJ\right\},
 \end{equation}
 Then 
\(F\,=A+JB\,\in End^J_{\mathbf{C}}(V)\) has the matrix representation   
 \[ A+JB\,=\, \left(
\begin{array}{cc}A&-\overline B\\[0.2cm]B&\overline A\end{array}\right)\,.\]

The isomorphism (\ref{qtomj2}) yields the following \(\mathbf{R}\)-linear isomorphism
 \begin{equation}
\, End_{\mathbf{H}}( V)\,\ni \,A+JB\,\,\stackrel{\simeq}{\longmapsto}\,\left(
\begin{array}{cc}A&-\overline B\\[0.2cm] B&\overline A\end{array}\right)\,\in\,End^J_{\mathbf{C}}(V)\,. 
 \end{equation}
 
We shall write  
\begin{eqnarray}\label{mj2nc}
 MJ(2n,\mathbf{C})&=&\left\{Z\in \mathfrak{gl}(2n,\mathbf{C})\,,\quad JZ=\overline ZJ\right\}\\[0.2cm]
&=& \left\{Z=\,\left(\begin{array}{cc}A& -\overline B\\[0.2cm]
B&\overline A\end{array}\right)\,;\quad A,\,B\in \mathfrak{gl}(n,\mathbf{C})\,\right\}\,.\nonumber
\end{eqnarray}
Then 
 \(\,\mathfrak{gl}(n,\mathbf{H})\) and \( MJ(2n,\mathbf{C})\) are isomorphic as matrix algebras over \(\mathbf{R}\,\).      \( MJ(2n,\mathbf{C})\) is also \(\mathbf{R}\)-algebra isomorphic to \(MJ(2,\mathbf{C})\otimes_{\mathbf{C}}\mathfrak{gl}(n,\mathbf{C})\) .   In fact the isomorphism given by 
   the following transformation of \(2n\times 2n \)-matrices: 
 \[\begin{array}{ccc} \, & (j) &\,\\
 \begin{array}{c}\,\\[0.2cm] 
 (i)\\[0.2cm] \,\end{array}&
 \left(\begin{array}{ccc}\cdots &\cdots \cdot & \cdots \\[0.2cm]
  \cdots & \left(\begin{array}{cc}a_{ij} &-\overline b_{ij}\\ b_{ij} &\overline a_{ij}\end{array}\right) &\cdots\\[0.2cm]
 \cdots &\cdots &\cdots\end{array}\right)
 &\,\end{array}
 \,\longrightarrow\, 
 \begin{array}{c}\,\\[0.2cm]
 (i)\\[0.2cm]
 \,\\[0.2cm]\,\\[0.2cm]
 (n+i)\\[0.2cm]
 \,\end{array}
 \left( \begin{array}{cccccc}\cdots&\cdots&\cdots&\cdots&\cdots&\cdots\\[0.2cm]
 \cdots& a_{ij}&\cdots&\cdots&-\overline b_{ij}&\cdots\\[0.2cm]
 \cdots&\cdots&\cdots&\cdots&\cdots&\cdots\\[0.2cm]
 \cdots&\cdots&\cdots&\cdots&\cdots&\cdots\\[0.2cm]
 \cdots& b_{ij}&\cdots&\cdots&\overline a_{ij}&\cdots\\[0.2cm]
 \cdots&\cdots&\cdots&\cdots&\cdots&\cdots\end{array}\right)\,.
 \]
 \begin{eqnarray}
& MJ(2,\mathbf{C})\otimes_{\mathbf{C}}\mathfrak{gl}(n,\mathbf{C})\,\ni\,
\sum_{i,j=1}^n\left\{\,a_{ij}E^{\prime}_{2i-1\,2j-1}\,+\,b_{ij}E^{\prime}_{2i\,2j-1}\,-\,\overline b_{ij}E^{\prime}_{2i-1\,2j}\,+\,
\overline a_{ij}E^{\prime}_{2i\,2j}\,\right\}\nonumber \\[0.2cm]
&\Longrightarrow \,\sum_{i,j=1}^n\,\left\{\,a_{ij}E_{ij}+\,b_{ij}E_{n+i,j}\,-\,\overline b_{ij}E_{i,n+j}+\overline a_{ij}E_{n+i,n+j}\right\}\in\, MJ(2n,\mathbf{C})\,,\label{coordinchang}
\end{eqnarray}
where \(E_{ij}\) is the \(n\times n\)-matrix with entry \(1\) at \((i,j)\)-place and \(0\) otherwise.   
Since \( MJ(2,\mathbf{C})\simeq\mathbf{H}\)  we have the following \(\mathbf{R}\)-algebra isomorphisms:
\begin{equation}
\mathfrak{gl}(n,\mathbf{H})\,\simeq\, MJ(2n,\mathbf{C})\simeq \mathbf{H}\otimes_{\mathbf{C}}\mathfrak{gl}(n,\mathbf{C})\,= \mathfrak{gl}(n,\mathbf{C})+J\mathfrak{gl}(n,\mathbf{C})\,.
\label{algiso}
\end{equation}

\subsection{Lie algebra \( \mathfrak{gl}(n,\mathbf{H})\) }

We define the following bracket on \( \mathfrak{gl}(n,\mathbf{H})\):
 \begin{eqnarray}\label{qLiebra}
 \left[ X_1+JY_1,\,X_2+JY_2\,\right]\,&=&\,(X_1X_2-X_2X_1- \overline Y_1Y_2 + \overline Y_2Y_1)
 \nonumber \\[0.2cm]
\quad  &&+J(Y_1X_2-Y_2X_1+\overline X_1Y_2-\overline X_2Y_1)\,\label{bracket1}
 \end{eqnarray}
 for \(X_1+JY_1,\,X_2+JY_2\in\mathfrak{gl}(n,\mathbf{H})\), \(X_i,\,Y_i\in \mathfrak{gl}(n,\mathbf{C}),\,i=1,2\).    It gives a real Lie algebra structure on \( \mathfrak{gl}(n,\mathbf{H})\).   
 More conveniently, by the basis \(\{E_{ij}\}_{i,j}\) of \(\mathfrak{gl}(n,\mathbf{C})\) we have
  \begin{equation*}
 \left[\,z_1\otimes E_{ij}\,,\,z_2\otimes E_{kl}\,\right]\,=\, (z_1z_2)\otimes \delta_{jk}E_{il}\,-\,(z_2z_1)\otimes \delta_{il}E_{kj}\,
 \end{equation*}
 for \(z_1, z_2\in\mathbf{H}\).     
It is easy to see that thus defined \(\mathbf{R}\)-linear bracket satisfies the antisymmetry equation and the Jacobi identity.   
  The bracket is evidently invariant under complex and quarternion conjugations: 
\begin{equation}
\sigma \,[X.Y]\,=\,[\,\sigma X\,,\sigma\,Y\,]\,, \quad \tau \,[X.Y]\,=\,[\,\tau X\,,\tau\,Y\,]\,.
\end{equation}
The eigensubspace of the involution \(\sigma\) corresponding to the eigenvalue \(+1\) 
is nothing but the complex Lie algebra \( \mathfrak{gl}(n,\mathbf{C})\).

\subsection{quarternion Lie algebras} 

 Let \((V,J)\) be a  \(\mathbf{H}\)-module.       Let \(\sigma\) be the \(\mathbf{C}\)-linear involution on \(V\) that  anti-commutes \(J\): \(\,J\sigma=\,-\sigma J\), and let \(\tau\) be the conjugate \(\mathbf{C}\)-linear involution on \(V\) that commutes with \(\sigma\): \(\sigma\tau=\tau\sigma\).   Let \(V_o\) be the \(\mathbf{C}\)-eigensubspace of \(\sigma\) corresponding to the eigenvalue \(+1\).   Then \(V=V_o+JV_o\,\).   \(\,V_o\) and \(JV_o\) are invariant under \(\tau\). 
 
 \begin{definition}\label{ql}~~~
    Let  \(\mathfrak{g}\) be a \(\mathbf{R}\)-submodule of the \(\mathbf{H}\)-module \((V,J)\).       \(\,\mathfrak{g}\,\) is called a {\it quarternion Lie algebra } if \(\,\mathfrak{g}\) is equipped with a bracket \([\,\cdot\,,\,\cdot\,]\) that satisfies the following properties:
 \begin{enumerate}
 \item
 \(\left(\,\mathfrak{g}\,,\,\left[\,,\,\right]\,\right)\) is a real Lie algebra:
 \begin{enumerate}
 \item
 The bracket operation is \(\mathbf{R}\)-bilinear.
 \item
 \[
 [\,X\,,\,Y\,]\,+\, [\,Y\,,\,X\,] =0 \qquad\mbox{ for all \(X,Y\in \mathfrak{g}\)}.\]
 \item
 \[[\,X\,,\,[\,Y,\,Z]\,]\,+\,[\,Y\,,\,[\,Z,\,X]\,]\,+\,[\,Z\,,\,[\,X,\,Y]\,]\,=\,0\qquad \mbox{ for all \(\,X,Y,Z\,\in \,\mathfrak{g}\)}.\]
 \end{enumerate}
  \item
 \(\sigma\) and \(\tau\) are homomorphisms of Lie algebra \(\mathfrak{g}\,\):      
\begin{enumerate}
\item
\(\left(\,\mathfrak{g}\,, \,\left[\,,\,\right]\,\right)\) is invariant under the involutions \(\sigma\) and \(\tau\).
 \item
  \[\sigma [\,X\,,\,Y\,]\,=\,[\,\sigma X\,,\,\sigma Y\,\,]\,, \quad \tau[\,X\,,\,Y\,]\,=\,[\,\tau X\,,\,\tau Y\,\,]\,,\quad \forall X,Y\in \,\mathfrak{g}.\]
  \end{enumerate}
 \end{enumerate}
\end{definition}

For a quarternion Lie algebra \(\mathfrak{g}\) we denote by   
\(
\mathfrak{g}^{\pm}\) the eigensubspace of the involution \(\sigma\) with the eigenvalue \(\pm 1\) respectively.     \(\mathfrak{g}^{\pm}\) is a vector subspace of \(
\mathfrak{g}\) invariant under the complex conjugation \(\tau\), and 
 \begin{equation}\label{pmLie}
 \mathfrak{g}=\mathfrak{g}^++\mathfrak{g}^-\,,\quad \mathfrak{g}^+=\mathfrak{g}\cap V_o\,, \quad \mathfrak{g}^-=\mathfrak{g}\cap JV_o\,.
 \end{equation}
  \(\mathfrak{g}^+\) becomes a subalgebra of \(\mathfrak{g}\).
  
 \begin{definition}\label{qf}~~~
 Let  \((\,\mathfrak{g}_o\,,\,\left[\,,\,\right]_o\,\,)\) be a complex or real Lie algebra.   Let 
  \((\, \mathfrak{g}\,,\left[\,,\,\right]\,)\) be a quarternion Lie algebra.      \(\,\mathfrak{g}\)  
is called the {\it quarternification of} \(\,\mathfrak{g}_o\,\)  if  \(\mathfrak{g}_o\) is a ( real ) Lie subalgebra of \(\mathfrak{g}^+\) 
and if there is a ( real ) vector subspace  \(\mathfrak{b}\) of \(\,\mathfrak{g}^-\) such that 
\(\,\mathfrak{g}_o+\mathfrak{b}\) generates \(\mathfrak{g}\) as a real Lie algebra. 
\end{definition}

Let \( \mathfrak{g}\) and \(\mathfrak{g}^{\prime}\) be quarternion  Lie algebras.   
A  homomorphism  \(\varphi\,:\,\mathfrak{g}\longrightarrow\,\mathfrak{g}^{\prime}\,\) of real Lie algebras is  called a {\it homomorphism of quarternion Lie algebras} if 
\begin{equation}
\,\varphi\,(\tau X\,)\,=\,\tau\,\varphi(X)\,\mbox{ and }\quad \varphi\,(\sigma X\,)\,=\,\sigma\,\varphi(X)\,,\quad\mbox{ for }\,\forall X\in\mathfrak{g}.
\end{equation}
We note that the quarternification is uniquely determined up to isomorphisms.

\begin{definition}~~\\
Let \(\mathfrak{g}\) be a quarternion Lie algebra and let  \(\mathfrak{p}\) be an ideal of \(\mathfrak{g}\) viewed as a real Lie algebra.     \(\mathfrak{p}\)  is called an  {\it ideal  of quarternion Lie algebra \(\mathfrak{g}\)}  if \(\mathfrak{p}\) is invariant under the involution \(\sigma\). 
\end{definition}

The quotient space of a quarternion Lie algebra \(\mathfrak{g}\) by an ideal \(\mathfrak{p}\) is endowed with a quarternion Lie algebra structure, where  the  involution  
\(\widehat{\sigma}\) on \( \,\mathfrak{g}/\mathfrak{p}\)  is defined by 
\[\widehat{\sigma}(x+\mathfrak{p})=\sigma x+\mathfrak{p}\,.\]

For a homomorphism of quarternion Lie algebra \(\varphi: \mathfrak{g}\longrightarrow\,\mathfrak{g}^{\prime}\), the kernel \(ker\,\varphi\) becomes an ideal of \(\mathfrak{g}\).

\begin{remark}\label{qmremark}~~~
Here is a remark on our abbreviation.      
 Let \(\mathfrak{g}\) be a quarternion Lie algebra and \(\mathfrak{g}=\mathfrak{g}^++\mathfrak{g}^-\) be the eigensapce decomposition by \(\sigma\); (\ref{pmLie}).   
Let \(\mathfrak{a}_o\) be a complex submodule of \(\mathfrak{g}^+\) and \(\mathfrak{a}_r\) be the real form: \(\mathfrak{a}_o=\mathfrak{a}_r+\sqrt{-1}\mathfrak{a}_r\).       
   Let \(\mathfrak{b}\) be a real Lie subalgebra  of \(\mathfrak{g}\) that is generated over \(\mathbf{R}\) by \(\mathfrak{a}_r+\sqrt{-1} \mathfrak{a}_r+J\mathfrak{a}_r+\sqrt{-1} J\mathfrak{a}_r\).     We abbreviate to call \(\mathfrak{b}\) {\it a quarternion subalgebra of \(\mathfrak{g}\) generated by \(\,\mathfrak{a}_o+J\mathfrak{a}_o\,\)}, though  \(\mathfrak{b}\) may not be a \(\mathbf{C}\)-module. 
  For example, \(\mathfrak{sl}(n,\mathbf{H})\) is {\it a quarternion Lie algebra generated by  \(\,\mathfrak{sl}(n,\mathbf{C})+\,J\mathfrak{sl}(n,\mathbf{C})\)}, though  \(\mathfrak{sl}(n,\mathbf{H})\) is not a \(\mathbf{C}\)-module.   In fact, for \(\mathfrak{a}_o=\mathfrak{sl}(n,\mathbf{C})\), we have 
 \(\mathfrak{a}_r=\mathfrak{sl}(n,\mathbf{R})\) and  \(\mathfrak{b}=\mathfrak{sl}(n,\mathbf{H})\).   \(\mathfrak{b}\) becomes 
   \begin{eqnarray}
   \mathfrak{b}&=&\mathfrak{a}_o+J\mathfrak{a}_o+\sqrt{-1}\mathbf{R}E_{nn}+J\mathbf{C}E_{nn}
    \label{decomp}\\[0.2cm]
   &=&\mathfrak{a}_r+\sqrt{-1}( \mathfrak{a}_r+\mathbf{R}E_{nn})+J( \mathfrak{a}_r+\mathbf{R}E_{nn})+J(\sqrt{-1} \mathfrak{a}_r+\sqrt{-1}\mathbf{R}E_{nn})\,.\nonumber
 \end{eqnarray}

 \end{remark}

\subsection{ Examples}

\begin{enumerate} 
\item
{\it 
\(\,\mathfrak{gl}(n,\mathbf{H})\) is a quarternion Lie algebra that is the quarternification of  \(\mathfrak{gl}(n,\mathbf{C})\).    }\\
We have allreday discussed it in 2.3 and (\ref{mj2nc}):
\begin{eqnarray}
\mathfrak{gl}(n,\mathbf{H})&=&\left\{X\in \mathfrak{gl}(2n,\mathbf{C})\,:\, \overline X\,J=JX\,\right\}
\\[0.2cm]
&=& \left\{\left(\begin{array}{cc}A& -\overline B\\[0.2cm]
B&\overline A\end{array}\right)\,;\quad A,\,B\in \mathfrak{gl}(n,\mathbf{C})\,\right\}\,.\nonumber
\end{eqnarray}

\item
{\it 
\(\mathfrak{sl}(n,\mathbf{R})+J \mathfrak{gl}(n,\mathbf{R})\) is a quarternion Lie subalgebra of \(\mathfrak{gl}(n,\mathbf{H})\).  }\\
 Here we admit the trivial \(\mathbf{R}\)-linear action of \(\tau\).    
There is a \(\mathbf{R}\)-algebra isomorphism; \(\,\mathfrak{sl}(n,\mathbf{R})+J \mathfrak{gl}(n,\mathbf{R})\ni A+JB \stackrel{\simeq }{\longrightarrow} A+\sqrt{-1}B\in \mathfrak{sl}(n,\mathbf{C})\,\).    The latter may be viewed as a quarternion Lie algebra by the trivial action of  \(\sigma\).

 \item
{\it \(\mathfrak{so}^{\ast}(2n)\) is a quarternion Lie algebra.}\\
\(\mathfrak{so}^{\ast}(2n)\)  is the Lie algebra of \(SO^{\ast}(2n)=SL(n,\mathbf{H})\cap O(2n,\mathbf{C})\), \cite{He}:
\begin{eqnarray*}
\mathfrak{so}^{\ast}(2n)&=&\left\{ X\in\mathfrak{gl}(n,\mathbf{H})\,:\, \,^tX+X=0\,\right\}
\,=\,\mathfrak{su}^{\ast}(2n)\cap \mathfrak{so}(2n,\mathbf{C})\\[0.2cm]
&\simeq &
\left\{\left(\begin{array}{cc}A& -\overline B\\ B & \overline A\end{array}\right)\in \mathfrak{sl}(2n,\mathbf{C});\, 
\begin{array}{cc} A&\,\in \mathfrak{so}(n,\mathbf{C}),\,
 \\ B&:\mbox{ Hermitian matrix}
\end{array} \,\right\}.
\end{eqnarray*}
The above \(\mathbf{R}\)-isomorphism is described by the change of matrix representations  
 from \(\mathbf{H}\otimes_{\mathbf{C}}\mathfrak{gl}(n,\mathbf{C})\) to 
\(MJ(2n,\mathbf{C})\), (\ref{coordinchang}).    We note that 
 \(JB\sim \left(\begin{array}{cc}0&-\overline B\\ B&0\end{array}\right)\in \,\mathfrak{so}(2n,\mathbf{C})\) iff  \(B\) is a Hermitian \(n\times n\)-matrix.   
  So \(\mathfrak{so}^{\ast}(2n)\) is a quarternification of  \(\mathfrak{so}(n,\mathbf{C})\).
 \item
{\it \(\mathfrak{sp}(n)\) is a quarternion Lie algebra.}\\
By virtue of the change of matrix representations due to the change of coefficients from \(\mathbf{H}\) to \(\mathbf{C}\), we have the \(\mathbf{R}\)-algebra  isomorphism:
\begin{eqnarray*}
\mathfrak{sp}(n)&=&\left\{X\in \mathfrak{gl}(n,\mathbf{H})\,:\,X^{\ast}+X=0\,\right\}
=\mathfrak{sp}(n,\mathbf{C})\cap \mathfrak{u}(2n)
\\[0.2cm]
&\simeq &\left\{\left(\begin{array}{cc}A&- \overline B\\ B & \overline A\end{array}\right)\in \mathfrak{sl}(2n,\mathbf{C});\, 
\begin{array}{ccc} A&:&\mbox{ skew Hermitian \(n\times n\) 
 matrix}
 \\ B&:&\mbox{ symmetric matrix}
\end{array} \,\right\}.
\end{eqnarray*}
Then 
\begin{equation}
\mathfrak{sp}(n)\,=\,\left\{ A+JB\in\mathfrak{gl}(n,\mathbf{H})\,;\,A\in \mathfrak{u}(n),\, B\in \mathfrak{s}\,\right\},
\end{equation}
 where \(\mathfrak{s}\) is the \(n\times n\) symmetric matrices.   Here we note that \(B\in\mathfrak{s}\) iff 
  \(JB\sim \left(\begin{array}{cc}0&-\overline B\\ B&0\end{array}\right)\in \,\mathfrak{u}(2n)\).   Hence \(\mathfrak{sp}(n)\) is a quarternification of the real Lie algebra \(\mathfrak{u}(n)\).   
  \item
{\it \(\mathfrak{sl}(n,\mathbf{H})\) is a quarternion Lie algebra.   It is the quarternification of  \(\mathfrak{sl}(n,\mathbf{C})\). }\\  
 We shall give a precise explanation of the quarternion Lie algebra \(\mathfrak{sl}(n,\mathbf{H})\)  in the next paragraph.   \(\mathfrak{gl}(n,\mathbf{H})\) can not be a quarternification of  \(\mathfrak{sl}(n,\mathbf{C})\).   
 
\item
The associative algebra generated by a \(\mathbf{H}\)-module has a natural quarternion Lie algebra structure.    Let \(V=\mathbf{H}\otimes_{\mathbf{C}}V_o=V_o+JV_o\) be a \(\mathbf{H}\)-module.   Let \(A_o\) be an associative algebra generated by the \(\mathbf{C}\)-module \(V_o\).   Then \(A=\mathbf{H}\otimes_{\mathbf{C}}A_o\) endowed with the multiplication rule defined by \((z_1\otimes \mathbf{v}_1)\cdot (z_2\otimes \mathbf{v}_2)=
(z_1z_2)\otimes (\mathbf{v}_1\cdot\mathbf{v}_2)\) becomes an associative \(\mathbf{C}\)-algebra  generated by the  \(\mathbf{H}\)-module \(V\).   The conjugate linear map \(J\) extends to an odd endomorphism of the algebra \(A\) , and we have  
 \( A\,=\,\mathbf{H}\otimes A_0\simeq A_o+JA_o \).
 Here we note that the product 
 \(
 J\mathbf{v}_1\cdot  J\mathbf{v}_2\,\cdots\,\cdot J\mathbf{v}_k\) belongs to \(A_o\) if \(k\) is even ( respectively  to  \(JA_o\) if \(k\) is odd ).   
The conjugation automorphisms \(\sigma\) and \(\tau\) on \(V=V_o+JV_o\) is equally extended to the conjugation automorphism on \(A=A_o+JA_o\).   \(A\) has naturally the quarternion Lie algebra structure defined as in  (\ref{qLiebra}), or equivalently:
\begin{equation}\label{qLiebra2}
[\,c_1\otimes \mathbf{v}_1\,,\,c_2\otimes \mathbf{v}_2\,]=(c_1c_2)\otimes (\mathbf{v}_1\cdot\mathbf{v}_2)-
(c_2c_1)\otimes (\mathbf{v}_2\cdot\mathbf{v}_1)\,, \end{equation}
for \(c_1,c_2\in\mathbf{H},\,\mathbf{v}_1,\mathbf{v}_2\in V_o\).

\item
Let  \(U(\mathfrak{g})\) be the universal enveloping algebra of \(\mathfrak{g}\), \cite{D}.   Then  \(\mathbf{H}\otimes_{\mathbf{C}}U(\mathfrak{g})\) is a quarternion Lie algebra, \cite{K-I}.
\end{enumerate}

\subsection{quarternion Lie algebra \( \mathfrak{sl}(n,\mathbf{H})\) }

Now we discuss \(\mathfrak{sl}(n,\mathbf{H})\) as the quarternification of \(\mathfrak{sl}(n,\mathbf{C})\).
By definition 
\begin{equation}
\mathfrak{sl}(n, \mathbf{H})\,=\,\left\{\,Z\in\,\mathfrak{gl}(n, \mathbf{H})\,;\quad Re.\, tr\,Z\,=0\,\right\}.
\end{equation}
We put 
\begin{equation*}
\mathfrak{sk}(n,\mathbf{C})=\{A\in \mathfrak{gl}(n,\mathbf{C})\,;\quad tr\,A\in \sqrt{-1}\mathbf{R}\,\},
\end{equation*}
then \( \mathfrak{sk}(n,\mathbf{C})\) is a real Lie algebra and 
\begin{equation*}
 \mathfrak{sl}(n,\mathbf{H})= \mathfrak{sk}(n,\mathbf{C})+\,J \mathfrak{gl}(n,\mathbf{C})=
 \left\{
 \left(\begin{array}{cc}A&-\overline B\\[0.2cm]
B&\overline A\end{array}\right)\in MJ(2n,\mathbf{C});\, 
\begin{array}{cc}A& \in \mathfrak{sk}(n,\mathbf{C}) \\
B&\in  \mathfrak{gl}(n,\mathbf{C})\end{array}
\right\}.
\end{equation*}
We have the following relation:
\begin{equation*}
\mathbf{H}\otimes_{\mathbf{C}} \mathfrak{sl}(n,\mathbf{C})\,
=  \,\mathfrak{sl}(n,\mathbf{C})+\,J\mathfrak{sl}(n,\mathbf{C})\,\subset \mathfrak{sk}(n,\mathbf{C})+\,J\mathfrak{gl}(n,\mathbf{C})=
\mathfrak{sl}(n,\mathbf{H})\,.
\end{equation*}  

The basis of the complex Lie algebra \(\mathfrak{g}_o=\mathfrak{sl}(n,\mathbf{C})\) is given by  
\[h_1=E_{11}-E_{22}, \,h_2=E_{22}-E_{33},\,\dots,\,h_{n-1}=E_{n-1,n-1}-E_{n,n},\quad
E_{ij}\,, \,i\neq j\,.\]
The first \((n-1)\) elements give a basis of diagonal matrices \(\mathfrak{h}_o\), which is the  Cartan subalgebra of  \(\mathfrak{g}_o\) of \(\dim \mathfrak{h}_o=n-1\).    Let \(\alpha_{ij}\,\) be 
the roots of \(\mathfrak{g}_o\) with respect to \(\mathfrak{h}_o\): \(\alpha_{ij}(h)=\lambda_i-\lambda_{j}\),\(\,h\in \mathfrak{h}_o\), \(i\neq j\,\).
The root space decomposition becomes
 \[\mathfrak{sl}(n,\mathbf{C})= \mathfrak{h}_o\oplus \sum_{i\neq j}\,(\mathfrak{g}_o)_{\alpha_{ij}}\,,\quad
 (\mathfrak{g}_o)_{\alpha_{ij}}=\mathbf{C}E_{ij}\,.\]
     The set of simple roots is \(\,\Pi=\{ \alpha_{i}\,;\,i=1,\cdots,n-1\,\}\),
 where we rewrite \(\alpha_i=\alpha_{i,i+1}\).   
  Fix a standard set of generators of \(\mathfrak{sl}(n,\mathbf{C})\):
      \[h_i\in \mathfrak{h}_o\,,\,e_i=E_{i\,i+1} \in (\mathfrak{g}_o)_{\alpha_i}\,, \,f_i =E_{i+1\,i}\in (\mathfrak{g}_o)_{-\alpha_i}\,,\quad 1\leq i\leq n-1\,,\]
so that 
 \begin{equation}\label{slrelation}
 [\,e_i,\,f_j \,]=h_j\delta_{ij}\,, \,[\,h_i,\,e_j\,]=e_j\,, \,[\,h_i,\,f_j\,]=-f_j\,.
 \end{equation}
\( \{E_{ij}\,;\,1\leq i\neq \, j\leq n\} \) are  generated as follows;
\begin{eqnarray}\label{compose}
E_{ij}&=&\left [\,e_i\,,\,[e_{i+1},\,\cdots\, [e_{j-2},\,e_{j-1}]\,]\cdots\,\right ],\,\mbox{ for } i<j\,,\nonumber\\[0.2cm]
E_{ij}&=&\,\left [\,f_j\,,\,[f_{j+1},\,\cdots\, [f_{i-2},\,f_{i-1}]\,]\cdots\,\right ],\,\mbox{ for } i>j\,.
\end{eqnarray}

The \(\mathbf{H}\)-module \( \mathbf{H}\otimes_{\mathbf{C}} \mathfrak{sl}(n,\mathbf{C})\,
=  \,\mathfrak{sl}(n,\mathbf{C})+\,J\mathfrak{sl}(n,\mathbf{C})\) has as its basis
\[B=\{\,h_i\,,\,Jh_i\,:\,1\leq i\leq n-1\,\}\cup \{\,E_{j\,k}\,,\,JE_{j\,k}\,:\, 1\leq j\neq \,k\leq n\,\}.\]

\begin{proposition}\label{slH}~~~
 \(\mathbf{H}\otimes_{\mathbf{C}} \mathfrak{sl}(n,\mathbf{C})=\mathfrak{sl}(n,\mathbf{C})+J\mathfrak{sl}(n,\mathbf{C})\) generates  the Lie algebra \( \mathfrak{sl}(n,\mathbf{H})\) over \(\mathbf{R}\).        The generators are given by 
\begin{equation}
h_i\,,\,e_i\,,\,f_i\,,\, Jh_i\,,\,Je_i\,,\,Jf_i\,; \quad ( \,i =1,\cdots,n-1\,)\,.   
\end{equation}
Hence the quarternification of \(\mathfrak{sl}(n,\mathbf{C})\) is  \( \mathfrak{sl}(n,\mathbf{H})\).
\end{proposition}
\begin{proof}~~~
 The basis \(B\) of \(\,\mathfrak{sl}(n,\mathbf{C})+\,J\mathfrak{sl}(n,\mathbf{C})\)
augmented by the two elements \(\sqrt{-1}\,E_{2\,2}\in\,\mathfrak{sk}(n,\mathbf{C})\setminus \mathfrak{sl}(n,\mathbf{C})\) and \(\sqrt{-1}\,J( E_{ii}+E_{nn})\in J(\mathfrak{gl}(n,\mathbf{C})\setminus \mathfrak{sl}(n,\mathbf{C}))\) present a basis of  \(\mathfrak{sl}(n,\mathbf{H})=\mathfrak{sk}(n,\mathbf{C})+J\mathfrak{gl}(n,\mathbf{C})\) ( as a \(\sigma\)-submodule of \(\mathfrak{gl}(n,\mathbf{H})\) ).   So we shall show that these two elements are generated by \(B\) ( as a real Lie algebra ).
In fact, we have  \(\,[\sqrt{-1}\,Jh_{1}\,,\,Jh_{2}\,]\,=\, -2\sqrt{-1}\,E_{2\,2}\), and  
  \([\,J(E_{ni}),\,\sqrt{-1}E_{in}\,]=\,\sqrt{-1}J( E_{ii}+E_{nn})\).   
We note that \(\mathfrak{sl}(n,\mathbf{H})\ominus (\mathfrak{sl}(n,\mathbf{C})+J\mathfrak{sl}(n,\mathbf{C}))\) 
belongs to the root space with root \(0\,\):   
\begin{equation*}
ad(h)\,E_{2\,2}\,=\,0\,, \quad ad(h)\,J(\,E_{i\,i}+E_{n\,n}\,)\,=0\,,\qquad\forall h\in \mathfrak{h}_o\,.
\end{equation*}
\end{proof}

\subsection{ Free quarternion Lie algebra }

Let  \(L_o\) be a Lie algebra over \(\mathbf{C}\) generated by a basis \(X\) of a \(\mathbf{C}\)-vector space \(V\).   We say that \(L_o\) is {\it free} on \(X\) if, given a mapping \(\phi_o\) of \(X\) into a complex Lie algebra \(\mathfrak{g}_o\), there exists a unique homomorphism \(\psi_o\,:\,L_o\longrightarrow \mathfrak{g}_o\) extending \(\phi_o\), \cite{B, H}.   There is a unique free Lie algebra \(L_o(X)\) generated by \(X\).

\begin{definition}~~\\
   Let \(X\) be a finite set that is identified with the subset \(1\otimes X\,\subset\,\mathbf{H}\otimes X\).   We denote  \(JX=j\otimes X\subset \,\mathbf{H}\otimes X\).   The set \((X ,\,JX)\) may be supposed to be a subset of a \(\mathbf{H}\)-module.      
Let   \(L\) be a quarternion Lie algebra generated by the basis 
\(\{X,\,JX\}\).   \(L\) is said to be {\it free on} \(\,\{X,\,JX\}\) if, given a  quarternion Lie algebra \(\mathfrak{g}\) and a \(\mathbf{Z}_2\)-graded mapping \(\phi:\,X+JX\,\mapsto \mathfrak{g}\),  there exists a unique homomorphism of quarternion Lie algebras  \(\psi:L\,\longrightarrow \mathfrak{g}\) that extends the mapping  \(\phi\).
\end{definition}

\begin{proposition}\label{freeLie}~~ 
There exists a unique free quarternion Lie algebra \(L\) on \(\,\{X,JX\}\) .   
\end{proposition}

{\it Proof.}~~~
Let \(V_o\) be a vector space over \(\mathbf{C}\) having \(X\) as basis.   Then \(V=\mathbf{H}\otimes_{\mathbf{C}}V_o=V_o+JV_o\) is a  \(\mathbf{H}\)-module with the basis \(\{X,JX\}\).  
Let \(A_o\) be the associative \(\mathbf{C}\)-algebra on the generators \(\{X\}\) and let \(A\) be the associative \(\mathbf{C}\)-algebra on the generators \(\{X,JX\}\).    
The set of all monomials in these generators form a basis of \(A\).     Such a generator has the form \(x_{i_1}\cdots x_{i_m}+J(y_{j_1}\cdots y_{j_n})\) with \(x_{i_1},\cdots, x_{i_m}, y_{j_1},\cdots , y_{j_n}\in X\).     The monomial 
 \(
 Jv_1\cdot  Jv_2\cdots\cdot Jv_k\) belongs to \(A_o\) if \(k\) is even, and belongs to  \(JA_o\) if \(k\) is odd.    Hence \(A=A_o+JA_o\) is a \(\mathbf{H}\)-module.     We denote by the same letter \(A\) the Lie algebra obtained from \(A\) by redefining the multiplication in the usual way; (\ref{qLiebra}).    \(A=A_o+JA_o\) becomes a quarternion Lie algebra and \(A_o\) is a complex Lie algebra that is a real Lie subalgebra of \(A\).      Let  \( L\) be the Lie subalgebra of  \(A\) that is generated by \(\{X,\,JX\,\}\) ( as a real Lie algebra ).   
 The Lie subalgebra  \(L_o=L\cap A_0\) is generated by \(X\).    We have 
 \(L=L_o+JL_o\) and \(L\) is  a quarternification of \(L_o\).      \(L\) is a quarternion Lie algebra as well as \(\mathbf{H}\)-module generated by \(\{X,\,JX\}\).    Given a \(\mathbf{Z}_2\)-graded mapping  \(\phi\) from \(\{X,JX\}\) into a quarternion Lie algebra \(\mathfrak{g}\).    \(\phi\) is extended first to a \(\mathbf{Z}_2\)-graded linear map \(V\,\mapsto\,\mathfrak{g}\,\), then canonically to 
 an associative algebra homomorphism \(\Phi:A\longrightarrow\,U(\mathfrak{g})\), and \(\Phi\) induces a homomorphism of quarternion Lie algebras: \(\Phi:\,A\longrightarrow U(\mathfrak{g})\).     
 \(\Phi\) restricts to give a  homomorphism of quarternion Lie algebras  \(\psi:\,L\longrightarrow \mathfrak{g}\,\).     Since the image of generators \(\{X,JX\}\) in \(\mathfrak{g}\) defines uniquely the mapping we have the uniqueness of \(\psi\,\).     
 \hfil\qed

\begin{definition}\label{qfication}
If \(L\) is a free quarternion Lie algebra on \(X=\{x_i;\,i\in \Lambda\}\), and if \(R\) is the ideal of \(L\) generated by elements 
\(\{\,f_j;\, j\in \Lambda^{\prime} \}\), we call \(L/R\)  the {\it quotient quarternion Lie algebra with generators \(\{x_i\}_{i\in \Lambda}\) and relations } \(\{f_i=0;\,i\in\Lambda^{\prime}\}\), where \(x_i\) are the images in \(L/R\) of the elements of \(X\).
\end{definition}

\section
{quarternification of a simple Lie algebra}

\subsection{Theorem of Chevalley, Harish-Chandra and Serre}~~~
In this part we summarize the well known procedure of giving a simple Lie algebra from its root system and Cartan matrix.   
Let \(\mathfrak{g}_o\) be a simple Lie algebra with Cartan matrix \(A=\left( c_{ij}\right )\).  Let \(\mathfrak{h}_o\) be a Cartan subalgebra, \(\Phi\) the corresponding root system.    Let \(\Pi=\{\alpha_i;\,i=1,\cdots,l=\dim\,\mathfrak{h}_o\}\subset \mathfrak{h}_o^{\ast}\) be the set of simple roots and  \(\{\alpha_i^{\vee}\,;\,i=1,\cdots,l\,\}\subset \mathfrak{h}_o\) be the set of simple coroots.   The Cartan matrix \(A=(\,c_{ij}\,)_{i,j=1,\cdots,l}\) is given by \(c_{ij}=\left\langle \alpha_i^{\vee},\,\alpha_j \right\rangle\).       
 \(\Pi\) is also a base of the real part \(\mathfrak{h}_{r}\) of \(\mathfrak{h}_o\).     So \(\alpha(h)\) is a real number for \(\forall\alpha\in \Phi\)  and \(\forall h\in\mathfrak{h}_{r}\).     Let  \(\mathfrak{g}_o= \mathfrak{h}_o\oplus \sum_{\alpha\in \Phi}\,(\mathfrak{g}_o)_{\alpha}\) be the root space decomposition with the root space  
\(\,(\mathfrak{g}_o)_{\alpha}=\{\xi\in\mathfrak{g}_o;\,ad(h)\xi\,=\,\alpha(h)\xi, \quad\forall h\in \mathfrak{h}_o\}\).   Then  \(\dim_{\mathbf{C}}\,(\mathfrak{g}_o)_{\alpha}=1\).     Let \(\Phi_{\pm}\) be the set of positive ( respectively negative )  roots of \(\mathfrak{g}_o\) and put 
\[\mathfrak{e}_o=\sum_{\alpha \in \Phi_{+}}\,(\mathfrak{g}_o)_{\alpha}\,,\quad \mathfrak{f}_o=\sum_{\alpha \in \Phi_{-}}\,(\mathfrak{g}_o)_{\alpha}\,.\]
Then we have the triangular decomposition \(\mathfrak{g}_o=\mathfrak{h}_o \oplus  \mathfrak{e}_o \oplus \mathfrak{f}_o\).  
   Fix a standard set of generators \(\,h_i\in \mathfrak{h}_o\,, \,e_i\in (\mathfrak{g}_o)_{\alpha_i}\),  \(f_i\in (\mathfrak{g}_o)_{-\alpha_i}\).      
\(\mathfrak{g}_o\) is generated by 
\(X\,=\, \{e_i,\,f_i,\,h_i\,;\,i=1,\cdots,l\,\}\), and these generators satisfy the relations:
\begin{equation}\label{S1}
[\,h_i,\,h_j\,]\,=\,0\,, \quad   [\,e_i\,,\,f_j\,] \,=\,\delta_{ij}h_i\,,\quad
  [\,h_i\,,\,e_j\,]\,=\,c_{ji}e_j\,,\quad
[\,h_i\,,\,f_j\,]\,=\,-\,c_{ji}f_j\,.
\end{equation}
This is a presentation of \(\mathfrak{g}_o\) by generators and relations which depend only on the root system \(\Phi\).     

Conversely given a Cartan matrix \(A\) there is a simple Lie algebra that is associated to \(A\), \cite{H, C}.    Below we shall explain briefly how to construct the simple Lie algebra corresponding to \(A\).     Later we follow this method to have our quarternification of a given simple Lie algebra.   
Let  \(\,\Phi\)  be a root system with a fundamental system \(\,\Pi=\,\{\alpha_1,\cdots\,,\alpha_l\,\}\,\), and let \(\,c_{ij}=\,<\alpha_i\,,\alpha_j>\, \) be the associated Cartan integers.    Let \(L_o\) be the free Lie algebra on \(3l\) generators
\(\,X\,=\, \{\,e_i\,,\, f_i\,,\, h_i\,;\quad 1\leq i\leq l\,\}\,.\)   
Let \(I_o\) be the ideal in \(L_o\) generated by the elements:
\begin{equation}\label{io} [\, h_i\,,\,h_j\,]\,,\quad [\, e_i\,,\,f_j\,]\,-\delta_{ij}\, h_i\,,\quad  [\,h_i\,,\,e_j\,]\,-\,c_{ji}\,e_j\,,\quad  [\, h_i\,,\, f_j\,]\,+\,c_{ji}\, f_j\,.
\end{equation}
 
\begin{theorem}[ Tits, Chevalley and Harish-Chandra]\label{CHC}~~ \(M_o=L_o/I_o\) is the complex Lie algebra with generators  \(\{\,e_i,\,f_i,\,h_i\,;\,1\leq l\leq l\,\}\) and relations (\ref{S1}).   The elements \(h_i\,; 1\leq i\leq l\,\), form a basis of a \(l\)-dimensional abelian subalgebra \(H_o\) of \(M_o\) and 
\[ M_o=F_o+H_o+E_o\,,\quad\mbox{ direct sum , } \]
where \(F_o\) ( respectively  \(E_o\) ) is the subalgebra of \(M_o\) generated by the \(f_1,\,\cdots,\, f_ l\,\) ( respectively \(e_1,\,\cdots,\,e_l\,\) ).
\end{theorem}
The proof of this theorem is given by constructing a suitable representation of
 \(M_o\), \cite{C, H, S}.      Let \(L_o^-\) be the free associative \(\mathbf{C}\)-algebra with generators \(\{\,f_1,f_2.\cdots,f_l\,\}\).    
Then \(L_o^-\) may be made into a \(M_o\)-module giving a representation 
\(\rho_o:\,M_o\longrightarrow\,End(\,L_o^-)\,\)defined by:
 \begin{eqnarray}\label{M0action}
 \rho_o(f_j)\, f_{i_1}\cdots f_{i_t}\,&=& \,f_jf_{i_1}\cdots f_{i_t} \nonumber\\[0.2cm]
 \rho_o(h_j)\,f_{i_1}\cdots f_{i_t}&=& -(c_{i_1j}+\cdots+c_{i_tj})
f_{i_1}\cdots f_{i_t}\nonumber \\[0.2cm]
\rho_o(e_j)\, f_{i_1}\cdots f_{i_t}&=& 
-\,\sum_{k=1}^t\delta_{ji_k}\left (\sum_{h=k+1}^t c_{ i_hj}\right)f_{i_1}\cdots \Check{f_{i_k}}\cdots f_{i_t}
\end{eqnarray}
where \(\Check{f_{i_k}}\) means that
 \(f_{i_k}\) is omitted from the product.    There is a unique extension to \(M_o\) of this action and we have a representation  \(\rho_o:\,M_o\longrightarrow \,End(L^-_o)\).    We find that the ideal \(I_o\) is in the kernel of the representation \(M_o\longrightarrow  \,End(L^-_o)\).   A detailed proof is at p.99 of \cite{C} or at p.97 of \cite{H}.
 
Put 
 \begin{equation}\label{adideal}
 e_{ij}=(ad\,e_i)^{-c_{ji}+1}(\,e_j\,)\,,\quad \,f_{ij}=(ad\,f_i)^{-c_{ji}+1}(\,f_j\,) \quad\mbox { for } i\neq j\,.
 \end{equation}
 Let \( \tilde I_o\) be the ideal of \(M_o\) generated by all \(e_{ij}\,,\,f_{ij}\,,\,1\leq i\neq j\leq\,l\,\).
 
 \begin{theorem}[Serre]\label{Serre}~~~ Let \(\Phi\) be a root system and let \(\mathfrak{g}_o=M_o/\tilde I_o\) be the Lie algebra generated by \(\{\,e_i,\,f_i,\,h_i\,;\,1\leq i \leq l\,\}\) that are subject to the relations (\ref{S1}) and the relations  
  \begin{equation}\label{S2}
  e_{ij}=f_{ij}=0\,;\,1\leq i\neq j\leq\,l\,.
  \end{equation} Then  \(\mathfrak{g}_o\) is a finite dimensional simple Lie algebra with the Cartan subalgebra spanned by  \(\{h_i; 1\leq i\leq l\}\), and with the corresponding root system \(\Phi\).
\end{theorem}

The proof is found at p.99 of \cite{H}.     It is based on the fact that the maps \(ad\,e_i\,:\,\mathfrak{g}_o\longrightarrow \mathfrak{g}_o\) and \(ad\,f_i\,:\,\mathfrak{g}_o\longrightarrow \mathfrak{g}_o\) are locally nilpotent and also on the facts that the root system \(\Phi\) is finite as well as the associated Weyl group that acts on \(\Phi\).

\subsection{quarternification of a simple Lie algebra }

  Let \(L_o\) be as in the preceding paragraph a free Lie algebra over \(\mathbf{C}\) on \(3l\) generators 
\(X=\{\,e_i\,,\,f_i\,,\,h_i\,;\quad 1\leq i\leq l\,\}\), 
and let \(I_o\) be the ideal of \(L_o\) generated by
\[ [\,h_i\,,\,h_j\,]\,,\quad [\,e_i\,,\,f_j\,]\,-\delta_{ij}\,h_i\,,\quad  [\,h_i\,,\,e_j\,]\,-\,c_{ji}\,e_j\,,\quad  [\,h_i\,,\,f_j\,]\,+\,c_{ji}\,f_j\,, \qquad (\ref{io})\]
where  \(c_{ij}\)'s are real numbers.   Let \(I_r\) be the real form of \(I_o\): \(I_o=I_r+\sqrt{-1}I_r\).
 
From Proposition \ref{freeLie} there is a free quarternion Lie algebra \(L\,\)  that is generated over \(\mathbf{C}\) by the following elements
\begin{equation}\label{freebase}
X+JX= \left\{\,e_i\,,\,f_i\,,\,h_i\,,\,  Je_i\,,\,Jf_i\,,\,Jh_i\,;\quad 1\leq i\leq l\,\right\}\,.
 \end{equation}
Then \(L=L_o+JL_o\) is a quarternion Lie algebra as well as a \(\mathbf{H}\)-module, and becomes the quarternification of the complex Lie algebra \(L_o\,\).   More precisely  
\[L=L_r+\sqrt{-1}L_r+JL_r+J(\sqrt{-1}L_r)\,,\]
where \(L_r\) is the real form of \(L_o\,\), see Remark \ref{qmremark}.   
 
Let \( I^{\ast}\) be the ideal of the real Lie algebra \(L\) 
 generated by the following  elements:
\begin{eqnarray}\label{I}
 \,  [\, h_i\,,\,Jh_j\,]\,,& \quad
  [\,J h_i\,,\,Jh_j\,]\,,
  \,
 \\[0.2cm]
[\,Je_i\,,\,f_j\,]\,-\delta_{ij}\,Jh_i\,, &\,
[\,e_i\,,\,Jf_j\,]\,-\delta_{ij}\,Jh_i\,,\,\quad
[\,Je_i\,,\,Jf_j\,]\,+\delta_{ij}\,h_i\,,\nonumber
\\[0.2cm]
 [\,h_i\,,\,Je_j\,]\,-\, c_{ji}\,Je_j\,,&\,
 [\,Jh_i\,,\,e_j\,]\,-\, c_{ji}\,Je_j\,,\,\quad
 [\,Jh_i\,,\,Je_j\,]\,+\,c_{ji}\,e_j\,,\nonumber
 \\[0.2cm]
  [\,h_i\,,\,Jf_j\,]\,+\, c_{ji}\,Jf_j\,,&\,
 [\,Jh_i\,,\,f_j\,]\,+\, c_{ji}\,Jf_j\,,\,\quad
 [\,Jh_i\,,\,Jf_j\,]\,-\,c_{ji}\,f_j\,.\nonumber
 \end{eqnarray}  
Let \(I=I_o+I^{\ast}\) be the sum of ideals \(I_o\) and \(I^{\ast}\).    \(I\,\) is generated by the elements (\ref{io}) and (\ref{I}).   
 Then we have  
 \begin{equation}
 I \,\subset L_o+JL_r\,,\quad \, I\cap\,  L_o\,=\,I_o\,.
 \end{equation}
 We note that the elements like \([\sqrt{-1}x_i, Jy_i]\)  and \([\sqrt{-1}Jx_i, Jy_i]\) for \(x_i, y_i=h_i,e_i,f_i\,\), are not necessarily in \(I\).   
 
 Let 
\(M=L/ I\) be the quotient algebra of \( L\) by 
\( I\).    Then \(M\) is the quarternification of \(M_o\).

We shall construct a suitable representation ( over \(\mathbf{R}\) ) of \(M\) so that we can study the Lie algebra \(M\) concretely.   The images of the generators (\ref{freebase}) of \(L\) will be written by the same notation:   
\( \{\,e_i\,,\, f_i\,,\,h_i\,,\,  J e_i\,,\,J f_i\,,\,J h_i\,,\quad 1\leq i\leq l\,\}\).   

Let \(L^-\) be the free associative quarternion Lie algebra generated by the basis \(\,\{f_i,\,Jf_i\,;\,\,i=1,2,\cdots,l\,\}\).
\begin{proposition}~~~
Let \(L_o\) be as before the free Lie algebra generated by  \(\{\,e_i\,,\,f_i\,,\,h_i\}\) and  let \(\rho_o:\,M_o\longrightarrow End(L_o^-)\) be the representation given by the actions (\ref{M0action}).    Then \(\rho_o:\,M_o\longrightarrow End(L_o^-)\) is extended to a representation \(\rho\,:\,M\longrightarrow End(L^-)\) that is given by the following formulae:
\begin{eqnarray}
\rho(h_j)( Jf_{i_1}\cdots f_{i_t})&=\rho(J h_j)\,(f_{i_1}\cdots f_{i_t})&= -({c_{i_1j}+\cdots+c_{i_tj}})
\,Jf_{i_1}\cdots f_{i_t}\label{Maction} \\[0.3cm]
\rho(f_j) (Jf_{i_1}\cdots f_{i_t})&= \rho(Jf_j)\,(f_{i_1}\cdots f_{i_t})&= 
Jf_jf_{i_1}\cdots f_{i_t}\nonumber\\[0.3cm]
\rho(e_j) (Jf_{i_1}\cdots f_{i_t})&=\rho(Je_j)\,(f_{i_1}\cdots f_{i_t})&=
-\sum_{k=1}^t\delta_{ji_k}\left (\sum_{h=k+1}^t c_{ i_hj}\right) Jf_{i_1}\cdots \Check{f_{i_k}}\cdots f_{i_t}\nonumber
\end{eqnarray}
\begin{eqnarray}
\rho(Jh_j)\, (Jf_{i_1}\cdots f_{i_t}))&=& (c_{i_1j}+\cdots+c_{i_tj})
\,f_{i_1}\cdots f_{i_t}\label{Maction2} 
\\[0.3cm]
\rho(Jf_j)\, (Jf_{i_1}\cdots f_{i_t})&=& 
-f_jf_{i_1}\cdots f_{i_t}\nonumber
\\[0.3cm]
\rho( Je_j)\, (Jf_{i_1}\cdots f_{i_t})&=& \,\sum_{k=1}^t\delta_{ji_k}\left (\sum_{h=k+1}^t c_{ i_hj}\right)f_{i_1}\cdots \Check{f_{i_k}}\cdots f_{i_t}\nonumber
\end{eqnarray}
\end{proposition}

{\it Proof}

Since \(f_1,\cdots,f_l,\,Jf_1,\cdots,Jf_l\,\) form a basis of \(L^-\,\) the above formulae determine uniquely the endomorphisms \(\,\rho(h_i), \rho(f_i),\rho( e_i),\, \rho(Jh_i), \rho(Jf_i)\) and  \(\rho(Je_i)\,\).   
Thus there is a unique homomorphism \(L\longrightarrow End(L^-)\) mapping \(\,X+JX\in L\,\) 
 to \(\,\rho(X+JX)\, \).   This induces a quarternion Lie algebra homomorphism \(L\longrightarrow End(L^-)\).   In order to obtain a homomorphism \(M\longrightarrow End(L^-)\) we must verify 
\( I\,\subset\, Ker\,\,\rho\,\).   
That is, \(\rho\) factors through \(M\), thereby making \(\,L^-\) an \(M\)-module.   The fact that \([h_i,h_j]\), \([e_i,\,f_j]-\delta_{ij}h_i\),  \([h_i,\,e_j]-c_{ji}e_j\) and \([h_i,f_j]+c_{ji}f_j\)  belong to the kernel of \(\rho\) has been mentioned in Theorem \ref{CHC} and the proof is found in Proposition 18.2 of \cite{H}.   We arrange the basis of \(L^-\) in such a way that \(\{\cdots, \,(f_{i_1}\cdots f_{i_s})\,,\,(J f_{i_1}\cdots f_{i_s})\,,\cdots, \,(f_{k_1}\cdots f_{k_t})\,,\,(J f_{k_1}\cdots f_{k_t})\,,\,\cdots,\,\cdots\}\) are in lexicographic order.   Then we see from (\ref{Maction}), (\ref{Maction2}) that \(h_j\) and \(Jh_j\) act diagonally multiplying each basis element of \(L^-\) by scalars, so that each pair of \(\,\rho(h_i)\) and \(\rho(Jh_j)\) commute, hence 
\([h_i,Jh_j]\), \([Jh_i,h_j]\) and \([Jh_i,Jh_j]\) belong to the kernel of \(\rho\).   

We have 
\begin{eqnarray*}
\rho(Je_i)\rho(f_j)\,f_{i_1}\cdots f_{i_t}\,&=&\,- \left (\sum_{h=1}^t c_{ i_hi}\right)Jf_{i_1}\cdots \cdots f_{i_t}\\[0.2cm]
&\,&-\sum_{k=1}^t\delta_{ii_k}\left (\sum_{h=k+1}^t c_{i_hi}\right)Jf_jf_{i_1}\cdots \Check{f_{i_k}}\cdots f_{i_t}
\end{eqnarray*}
\[\rho(f_j)\rho(Je_i)\,f_{i_1}\cdots f_{i_t}\,=\,-\sum_{k=1}^t\delta_{ii_k}\left (\sum_{h=k+1}^t c_{ i_hi}\right)Jf_jf_{i_1}\cdots \Check{f_{i_k}}\cdots f_{i_t}\]
Hence 
\[ (\rho(Je_i)\rho(f_i)-\rho(f_i)\rho(Je_i))f_{i_1}\cdots f_{i_t}\,=\,  -\left (\sum_{h=1}^t c_{ i_hi}\right)Jf_{i_1}\cdots \cdots f_{i_t}=\rho(Jh_i)f_{i_1}\cdots f_{i_t},\]
so that \([Je_i\,, \,f_i\,]-Jh_i\in Ker\,\rho\).  Also  \([Je_i\,, \,f_j\,]\,\in Ker\,\rho\) for \(i\neq j\).
Next 
\begin{eqnarray*}
&&(\rho(Jh_i)\,\rho(f_j)-\rho(f_j)\rho(Jh_i)) f_{i_1}\cdots f_{i_t}=\rho(Jh_i)(f_jf_{i_1}\cdots f_{i_t}) +\rho(f_j)(c_{i_1i}+\cdots+c_{i_ti})\,Jf_{i_1}\cdots f_{i_t}\,\\[0.2cm]
&&=
\,\left( -(c_{ji}+c_{i_1i}+\cdots+c_{i_ti})+(c_{i_1i}+\cdots+c_{i_ti})\right)
\,Jf_jf_{i_1}\cdots f_{i_t}\,=\,
-\,c_{ji} \,Jf_j\cdot f_{i_1}\cdots f_{i_t}\,.
\end{eqnarray*}
Therefore
\(
[\,Jh_i\,,\,f_j\,]\,+\, c_{ji} Jf_j\,\in Ker \,\rho\,
\).

To prove  
\([\,Jh_i\,,\,e_j\,]\,-\, c_{ji} Je_j\,\in Ker \,\rho\,\)  we have
\[\rho(Jh_i)\rho(e_j)f_{i_1}\cdots f_{i_t}=-\sum_{k=1}^t\delta_{ji_k}\left(\sum_{h=k+1}^tc_{i_hj}\right)\left(-\sum_{g\neq k}c_{i_gi}\right)\,Jf_{i_1}\cdots\Check{f_{i_k}}\cdots f_{i_t}\,.\]
\[\rho(e_j)\rho(Jh_i)f_{i_1}\cdots f_{i_t}=-\sum_{k=1}^t\delta_{ji_k}\left(\sum_{h=k+1}^tc_{i_hj}\right)\left(-\sum_{g=1}^tc_{i_gi}\right)\,Jf_{i_1}\cdots\Check{f_{i_k}}\cdots f_{i_t}\,.\]
Thus 
\begin{eqnarray*}
(\,\rho(Jh_i)\rho(e_j)-\rho(e_j)\rho(Jh_i)\,)f_{i_1}\cdots f_{i_t}\,&=&\,-c_{ji}\sum_{k=1}^t\delta_{ji_k}\left(\sum_{h=k+1}^tc_{i_hj}\right)\,Jf_{i_1}\cdots\Check{f_{i_k}}\cdots f_{i_t}\,\\[0.2cm]
&=&c_{ji}\rho(Je_j)\,f_{i_1}\cdots f_{i_t}\,.
\end{eqnarray*}
Similarly other elements of \(I\) are shown to be in the kernel of \(\,\rho\).     Thus we have a homomorphism \(M\longrightarrow\,End(\,L^-\,)\).   
\hfil\qed

We can deduce useful information about \(M\) from the existence of this homomorphism.

\begin{proposition}\label{h}
The elements \(h_1,\,\cdots\,,\,h_l\,\) and \(Jh_1,\,\cdots\,,\,Jh_l\,\) of \(M\) 
are linearly independent over \(\mathbf{C}\).
\end{proposition}
We shall show that both  \(\rho(h_1),\,\cdots\,,\,\rho(h_l)\,\) and  \(\rho(Jh_1),\,\cdots\,,\,\rho(Jh_l)\,\) of \(End(L^-)\) are linearly independent.     Then \(h_1,\,\cdots\,,\,h_l\,,\,Jh_1,\,\cdots\,,\,Jh_l\,\) are linearly independent.        Let \(m=\sum_{j=1}^l\,a_jh_j\in M_0\).      Then  if  \(\rho(m)=0\) we have  \(-\sum a_jc_{ij}=0\,\)  for \(\,1\leq \forall i\leq l\,\).   
   Since the Cartan matrix \(c_{ij}\) of \(\Phi\) is non-singular, all \(a_j=0\).    
Hence  \(\rho(h_1),\,\cdots\,,\,\rho(h_l)\,\) are linearly independent, and so \(h_1,\,\cdots\,,\,h_l\,\) are  linearly independent also.    Similarly \(n=\sum_{j=1}^l\, a_jJh_j\in Ker\,\rho\) yields also 
\(\sum a_jc_{ij}=0\), \(\,1\leq \forall i\leq l\,\).   Hence  \(\rho(Jh_1),\,\cdots\,,\,\rho(Jh_l)\,\) are linearly independent, so are \(Jh_1,\,\cdots\,,\,Jh_l\,\). 
\hfil\qed

\begin{definition}~~~
\begin{enumerate}
\item
Let \(H_o\) be the Cartan subalgebra of \(M_o\) with the basis \(h_i\,;\,1\leq i\leq l\).
\item
Let \(H_r\) be the real form of \(H_o\): \(H_o=H_r+\sqrt{-1}H_r\).    
\item
Let  \(K\,\) be the subalgebra  of \(M\) generated over \(\mathbf{R}\) by  \(\,\mathbf{H}\otimes_{\mathbf{C}} H_o=H_o+JH_o\).   
\item
Let \(H^{\bot}_r\) be the subalgebra of \(K\) generated by \(\sqrt{-1}H_r+JH_o\).
\end{enumerate}
\end{definition}

\begin{lemma}\label{K}~~~
\begin{enumerate}
\item
\(H_r\) is a maximal commutative subalgebra of \(M\).     
\item
\(H^{\bot}_r\) is an ideal of \(K\) complementary to \(H_r\).
\item
We have 
\begin{equation}
 [\,H_r\,,\,K\,]=\,0\,.
 \end{equation}
\item
\begin{equation}\label{1bracket}
[\,K\,,\,K\,]\,=\,H^{\bot}_r\,.
\end{equation}
\end{enumerate}
\end{lemma}
\begin{proof}~~~
 Being the real part of Cartan subalgebra \(H_o\,\), \(H_r\,\) is commutative.    While for any \(x\in \sqrt{-1}H_r+JH_o\subset K\) there exists a \(y\in K\) which does not commute with \(x\).     Hence there is no commutative subalgebra of \(K\) bigger than \(H_r\),  nor in \(M\). 
 We shall prove the 3rd assertion.   We have 
\([H_r, H_o]=0\,\) and \([ \ H_r,\,J H_o\,]=0\,\).
The latter follows from \([rh_i\,,\,cJh_j\,]=rc[\,h_i,\,Jh_j\,]=0\) for \(\forall r\in \mathbf{R},\, \forall c\in\mathbf{C}\).
Jacobi identity yields \([H_r,\,[H_o,JH_o]]=0\) and \([H_r,\,[JH_o,JH_o]]=0\).   Next it holds for example that 
\[[H_r,\,[\,[JH_o,H_o],\,JH_o]]=0,\mbox{ and }
[H_r,\,[\,[JH_o,\,JH_o],\,H_o]]=0\,.\]
In this way we obtain \([\,H_r\,,\, K\,]\,=\,0\,\).
Then \([H_r^{\bot},\,K]=[H^{\bot}_r,\,H^{\bot}_r]=H^{\bot}_r\), so \(H^{\bot}_r\) is an ideal of \(K\).   
 This implies the 4th assetion.    
  \end{proof}
  
{\bf Example.}~~~\\
For \(M=\mathfrak{sl}(n,\mathbf{H})\) and \(M_o=\mathfrak{sl}(n,\mathbf{C})\) we have 
\begin{eqnarray*}
H_r\,&=&\sum_{i=1}^{n-1}\mathbf{R}h_i\,=\sum_{i=1}^{n-1}\mathbf{R}(E_{i+1\,i+1}-E_{i\,i})\,,\quad K= H_o+\sqrt{-1}\mathbf{R}E_{1\,1}+J (H_o+\mathbf{C}E_{1\,1})
\\[0.2cm]
H^{\bot}_r &=&\,\sqrt{-1}\sum_{i=1}^{n}\mathbf{R}E_{i\,i}\,+\,J(H_o+\mathbf{C}E_{1\,1})\,,
.\end{eqnarray*}
where \(H_o\) is the Cartan subalgebra of \(\mathfrak{sl}(n,\mathbf{C})\) and \(H_r\) is the real form of \(H_o\).   
This was shown in the proof of Proposition \ref{slH}.

\begin{theorem}\label{quatofsimpleLie}~~~
\begin{enumerate}
\item
 \(M\) is a quarternion Lie algebra generated by  
 \[\{\,e_i,\,f_i,\,h_i\,, \,Je_i,\,Jf_i,\,Jh_i\,;\quad 1\leq l\leq l\,\}\]
 that are subordinate to the relations
 \begin{equation*}
 [\,h_i,\,h_j\,]=0\,, \quad   [\,e_i\,,\,f_j\,] =\delta_{ij}h_i\,,\quad
  [\,h_i\,,\,e_j\,]=c_{ji}e_j\,,\quad
[\,h_i\,,\,f_j\,]=-\,c_{ji}f_j\,. \quad  (\ref{S1})
\end{equation*}
and
\begin{eqnarray}\label{relation2}
  [\,h_i,\,Jh_j\,]=0,\,& \quad  [\,Jh_i,\,Jh_j\,]=0\,,  & \quad [\,h_i\,,\,Jf_j\,]\,=\,-c_{ji}\,Jf_j\,, 
 \nonumber\\[0.2cm]
  [\,h_i\,,\,Je_j\,]\,=\,c_{ji}\,Je_j\,, 
&\quad [\,Jh_i\,,\,f_j\,]\,=\,- c_{ji}\,Jf_j\,, & \quad   [\,Jh_i\,,\,e_j\,]\,=\,c_{ji}Je_j\,,
\nonumber\\[0.2cm]
 [\,e_i \,,\,Jf_i\,] =\,\delta_{ij}\,Jh_i, &\quad  [\,Je_i\,,\,f_i\,] =\,\delta_{ij}Jh_i, &\quad   [\,Je_i\,,\,Jf_i\,] =-\,\delta_{ij}\,h_i \,,\nonumber\\[0.2cm]
   \,[\,Jh_i\,,\,Je_j\,]\,=\,-c_{ji}\,e_j\,,&\quad
[\,Jh_i\,,\,Jf_j\,]\,=\, c_{ji}f_j\,&\,.
\end{eqnarray}
\item
 \(M\) is the quarternification of  \(M_o\).
 \item 
 The real form \(H_r\) of the 
 Cartan subalgebra \(H_o\) of \(M_o\) is a maximal commutative subalgebra of \(M\).    
 \item
 Let \(E\) be the \(\mathbf{C}\)-module generated by  \(\{e_i,\,Je_i\,;\,i=1,\cdots,l\,\}\), and  \(F\) be the \(\mathbf{C}\)-module generated by  \(\{f_i,\,Jf_i\,;\,i=1,\cdots,l\,\}\).   Then \(K\), \(E\)
and \(F\) viewed as real Lie subalgebras of \(M\) give the triangular decomposition of \(M\):
\begin{equation}
M=K+E+F\,.
\end{equation}   
 
 \end{enumerate}
\end{theorem}

{\it Proof}~~~
From the definition \ref{qfication}  \(M\) is the quarternification of \(M_o\).    
The theorem will be proved by the well known argument that is used to prove Theorem \ref{CHC}.   
\begin{enumerate}
\item \(K\) {\it is isomorphically embedded in} \(M\).\\
This follows from Proposition \ref{h}
\item
{\it The subspace \(\sum\,\mathbf{C}e_j\,+\,\sum\,\mathbf{C}f_j\,+\,\sum\,\mathbf{C}h_j\) of \(L\) maps isomorphically into \(M\).   Hence \(M_o\) is a subalgebra of \(M\). }\\
 For a fixed \(i\), \(\mathbf{C}e_i+\mathbf{C}f_i+\mathbf{C}h_i\) is  a homomorphic image of \(\mathfrak{sl}(2,\mathbf{C})\).  The latter being simple and  \(h_i\neq 0\) in \(M\) by Proposition \ref{h}, 
so \(\mathbf{C}e_i+\mathbf{C}f_i+\mathbf{C}h_i\) is isomorphic to \(\mathfrak{sl}(2,\mathbf{C})\).
Now the set \(\{e_j,\,f_j,\,h_j\,;\, 1\leq j\leq l\,\}\) is linearly independent as the eigenvectors of the \(ad\,h_j\), so the subspace \(\sum\,\mathbf{C}e_j\,+\,\sum\,\mathbf{C}f_j\,+\,\sum\,\mathbf{C}h_j\) of \(L\) maps  isomorphically into \(M\).   
\item
{\it Let \(\epsilon_i\) denote the basis \(e_i\) or \(Je_i\) and let \(\phi_i\)  denote the basis \(f_i\) or \(Jf_i\).    
 We write 
\begin{eqnarray*}
[\,\epsilon_{i_1}\,\cdots\,\epsilon_{i_t}\,] \,&=\,[\epsilon_{i_1},\,[\epsilon_{i_2},\,\dots,\,[\epsilon_{i_{t-1}},\,\epsilon_{i_t}\,]\,]\, ]\\[0.2cm]
[\,\phi_{i_1}\,\cdots\,\phi_{i_t}\,] \,&=\, [\phi_{i_1},\,[\phi_{i_2},\,\dots,\,[\phi_{i_{t-1}},\,\phi_{i_t}\,]\,]\, ]\,.
\end{eqnarray*}
 Then }
\begin{eqnarray}
[\,h_j\,,\,[\,\epsilon_{i_1}\,\cdots\,\epsilon_{i_t}\,] \,]\,&=\,\left (c_{i_1j}+\cdots+c_{i_tj}\right)\,[\,\epsilon_{i_1}\,\cdots\,\epsilon_{i_t}\,] ,\label{adh}
\\[0.2cm]
[\,h_j\,,\,[\,\phi_{i_1}\,\cdots\,\phi_{i_t}\,] \,]\,&=\,-\left (c_{i_1j}+\cdots+c_{i_tj}\right)\,[\,\phi_{i_1}\,\cdots\,\phi_{i_t}\,] .\label{adjh}
\end{eqnarray}
 For \(t=1\) we have \([h_j,e_i]=c_{ij}e_i\) and  \([h_j,Je_i]= c_{ij}Je_i\) in \(M\).   We have, for example,  
\[[h_j,[e_i\,,Je_k]\,]\stackrel{{\rm Jacobi}}{=}[e_i,[h_j,Je_k]]+[[h_j,e_i],Je_k]=(c_{ij}+ c_{kj})[e_i,Je_k].\]
The general case follows by induction and the Jacobi identity.   

\item
{\it Let \(E\) and \(F\) be the subalgebras of \(M\) generated by \(\{e_i,\,Je_i\,;\,i=1,\cdots, l\,\} \) and  by  \(\{f_i,\,Jf_i\,;\,i=1,\cdots, l\,\} \) respectively.   Then 
\begin{equation}
[\,Jh_j\,,\,[\,\epsilon_{i_1}\,\cdots\,\epsilon_{i_t}\,] \,]\,\in E
 ,\qquad \,[\,Jh_j\,,\,[\,\phi_{i_1}\,\cdots\,\phi_{i_t}\,] \,]\,\in \,F .
\end{equation}
}  
 By virtue of the relations \([Jh_j,e_i]=c_{ij}Je_i\) and  \([Jh_j,Je_i]=-c_{ij}e_i\),  we have   
\[[Jh_j, \,[e_i,[e_k,Je_l]]]=-c_{kj}[e_i,[e_k,e_l]]- c_{kl}[e_i,[e_k,\,e_l]]-c_{ij}\,J[e_i,[e_k,e_l]].\]
By similar calculations we obtain the first assertion.   
The other cases follow by an induction with the aid of the relations
 \([Jh_j,f_i]=-c_{ij}Jf_i\).
 
\item
{\it  Let \(\pi(i_1,\cdots, i_t)=\sharp\,\{s:\,\epsilon_{i_s}=Je_{i_s},\,1\leq s\leq t\}\).   The parity \(\pi([\,\epsilon_{i_1}\,\cdots\,\epsilon_{i_t}\,] )\) of \([\epsilon_{i_1}\,\cdots\,\epsilon_{i_t}] \in M\) is  defined by  \((-1)^{\pi(i_1,\cdots, i_t)}\).   It is equal to \(1\) ( respectively \(\,-1\)) if  there are  even ( respectively odd ) number of  \(\epsilon_{i_s}\)'s that are equal to \(Je_{i_s}\) among \(1\leq s\leq t\).   
Similarly the parity  \(\pi([\,\phi_{i_1}\,\cdots\,\phi_{i_t}\,])\) is defined.\\
Suppose \(t\geq 2\).  Then 
\([f_j\,,\,[\,\epsilon_{i_1}\,\cdots\,\epsilon_{i_t}\,] \,]\,\in E_o\)  if \(\pi\) is even, and \([f_j\,,\,[\,\epsilon_{i_1}\,\cdots\,\epsilon_{i_t}\,] \,]\,\in JE_o\)  if \(\pi\) is odd.     Similarly \([e_j\,,\,[\,\phi_{i_1}\,\cdots\,\phi_{i_t}\,] \,]\,\in F_o\)  if \(\pi\) is even and \([e_j\,,\,[\,\phi_{i_1}\,\cdots\,\phi_{i_t}\,] \,]\,\in JF_o\)  if \(\pi\) is odd.   }   \\
We have already known that \([f_j,\,[\,e_{i_1}\,\cdots\,e_{i_t}\,] \,]\,\in E_o\) and the parity of \([\,e_{i_1}\,\cdots\,e_{i_t}\,]\) is even.   
Next the parity of \([e_i,Je_k]\) is odd, and we have 
\begin{eqnarray*}
[f_j,\,[e_i,Je_k]]&=&[e_i,[f_j,Je_k]]+[[f_j,e_i],Je_k]=-[e_i,\delta_{kj}Jh_k]-[\delta_{ji}h_i,Je_k]\\[0.2cm]
&=&\delta_{kj} c_{ik}Je_i -\delta_{ji}  c_{ki}Je_k \in JE_o.
\end{eqnarray*}
By the relations \([f_j,e_i]=-\delta_{ij}h_i\),  \([f_j,Je_i]=-\delta_{ij}Jh_i\) and the Jacobi identity we complete the argument by an induction on \(t\).

\item
{\it \(K+E+F \) is  a subalgebra of \(M\), hence coincides with \(M\).}  \\
  That  \(K+E+F\) is  a subalgebra follows from the steps 3, 4, 5.     In fact, from (\ref{adh}) and (\ref{adjh}) we have 
  \([h_i,E] \subset E\) and \([Jh_i,E]\subset E\), hence 
  \([K,\,E]\subset E\).   It follows that \(K+E\) is a subalgebra of \(M\).    Similarly \(K+F\)  is a subalgebra of \(M\).    The assertion in step 5 shows that 
  \([\epsilon_i\,,\,F\,]\subset F+K\,\).   
Since \(K+E\) is a subalgebras it follows that 
\([\,\epsilon_i\,,\,F+K+E\,]\subset F+K+E\,\).   
Similarly we have 
\([\,\phi_i\,,\,F+K+E\,]\subset F+K+E\,\).   
The relations \([h_i, \,F+K+E\,]\subset F+K+E\,\) and \([Jh_i, \,F+K+E\,]\subset F+K+E\,\) are clear.   Therefore the set of all \(x\in M\) such that \([x,\,F+K+E]\subset F+K+E\) contains \(h_i,Jh_i, e_i, Je_i. f_i, Jf_i\).    However the set of such \(x\) forms a subalgebra.   This subalgebra must be the whole \(M\).   Thus 
\(F+K+E\) is an ideal of \(M\) that contains the generators of \(M\).    It follows 
that \(M=K+E+F\).  
\item
{\it The sum \(M=K+E+F\)  is direct.  }  \\
The step 3 shows how to decompose \(M\) ( directly ) into the eigensubspaces of \(ad\,H_r\), and the directness follows.   
\hfil\qed
\end{enumerate}

We consider now the weight spaces of \(M\) with respect to the maximal commutative subalgebra \(H_r\) and describe the decomposition \(M=K+E+F\) in terms of weights.    
 For each \(\lambda:\, H_r\longrightarrow  \mathbf{R}\,\), let 
 \begin{equation}
  M_{\lambda}=\{x\in M\,;\,[h\,, x\,]=\lambda(h)x, \quad \forall h\in H_r\}\,.
 \end{equation}
 Do not confuse \(M_0\) with \(M_o=L_o/I_o\).     \(\lambda\in Hom(\,H_r\,,\,\mathbf{R}) \) will be called a {\it weight} whenever  \(M_\lambda\neq 0\).    \(M_\lambda\) is called the {\it weight space} of weight \(\lambda\).   
  If \(x,\,y\in M\) are weight vectors of weights \(\lambda,\,\mu\) then \([x\,,\,y]\) is a weight vector of weight 
\(\lambda+\mu\):   
\[[\,M_\lambda\,,\,M_\mu\,]\subset\,M_{\lambda+\mu}\,.\]
Any root of \(M_o\) is a weight of \(M\) since \(\lambda(h)\) is real for any \(\lambda\in \Phi\) and \(h\in \mathfrak{h}_r\).   From Lemma \ref{K} and the directness of the decomposition of \(M\) it follows that the null weight space \(M_0\) coincides with \(K\).  
From the relations (\ref{S1}) and (\ref{relation2}) we see that \(e_i\) and \(Je_i\) are weight vectors of weight \(\alpha_i\).   \(f_i\) and \(Jf_i\) are weight vectors of weight \(-\alpha_i\).  
  Thus all Lie products of generators \(e_i,Je_i,\,f_i,Jf_i,\,h_i,Jh_i\,\) are weight vectors.    Since every element of \(M\) is a linear combination of products of these vectors we deduce that  the weights of \(M\) coincides with the root system \(\Phi\) of \(M_o\): 
\[M=\,K\oplus \sum_{\lambda\in \Phi}\,M_{\lambda}.\]
  We know already that the weights \(\alpha_1,\cdots,\alpha_l\in Hom(H_r,\mathbf{R})\) are linearly independent, and that 
any weight \(\lambda\in \Phi\) has form \(\sum_{i=1}^l\,k_i\alpha_i\), \(k_i\in\mathbf{Z}\).   Moreover a non-zero weight \(\lambda\) has the form \(\lambda=\sum_{i=1}^lk_i\alpha_i,\,k_i\in \mathbf{Z}\), with all \(k_i\geq 0\) or all \(k_i\leq 0\).    From the discussion hitherto we have 
\begin{equation}\label{tridecomp}
K=M_0\,,\quad 
E=\sum_{\lambda\in \Phi_+}\,M_\lambda\,,\quad F=\sum_{\lambda\in \Phi_-}\,M_\lambda\,.
\end{equation}

\subsection{Root space decomposition of a quarternion Lie algebra}
   
Let \(\hat I_o\) be the ideal of \(M_o\) generated by all 
\(
e_{ij}=(ad\,e_i)^{-c_{ji}+1}(e_{j})\) and \( f_{ij}=(ad\,f_i)^{-c_{ji}+1}(f_{j})\), (\ref{S2}).
We know that \(\mathfrak{g}_o=M_o/\hat I_o\)  is a finite dimensional simple algebra with Cartan subalgebra \(\mathfrak{h}_o\) which 
is the image of \(H_o\) under the quotient map \(M_o\longrightarrow\,\mathfrak{g}_o\).     In fact \(H_o\) maps isomorphically into \(\mathfrak{g}_o\).    \(\mathfrak{h}_o\) is spanned by \(h_1,\cdots,h_l\) and the corresponding root system is \(\Phi\) ,  
Theorem\ref{Serre}, \cite{C, H} .
The Cartan decomposition of \(\mathfrak{g}_o\) is given by
\begin{eqnarray*}
&\mathfrak{g}_o\,=\,\mathfrak{h}_o\oplus\,\sum_{\lambda\in\Phi} (\mathfrak{g}_o)_\lambda\,, \\[0.2cm]
&( \mathfrak{g}_o)_{\lambda}\,=\{\xi\in \mathfrak{g}_o;\,ad(h)\,\xi\,=\lambda(h)\xi\quad \forall h\in \mathfrak{h}_o\}\,.
\end{eqnarray*}

Let \(\hat I\) be the ideal of \(M\) generated by \(\mathbf{H}\otimes_{\mathbf{C}} \hat I_o=\hat I_o+J\hat I_o\), that is, the ideal generated by 
\(e_{ij},\,Je_{ij},\,f_{ij},\,Jf_{ij}\).    \(\hat I\) is in fact the ideal written by 
the elements \([\epsilon_i,[\epsilon_i,\cdots,[\epsilon_i,\epsilon_j]]]\)  and the elements \([\phi_i,[\phi_i,\cdots,[\phi_i,\phi_j]]]\) for all \(i\neq j\).   Where \(\epsilon_i=e_i\) or \(Je_i\), \(\phi_i=f_i\) or \(Jf_i\) as was introduced before.    Let \(\mathfrak{g}\) be the quotient quarternion Lie algebra:
\begin{equation}
\mathfrak{g}=M/\hat I\,.
\end{equation}
\(\mathfrak{g}\) is the quarternification of \(\mathfrak{g}_o=M_o/\hat I_o\).    Let \(\hat I^+\) be the ideal of \(E\) generated by the elements \([\epsilon_i,[\epsilon_i,\cdots,[\epsilon_i,\epsilon_j]]]\) for all \(i\neq j\), and let \(\hat I^-\) be the ideal of \(F\) generated by the elements \([\phi_i,[\phi_i,\cdots,[\phi_i,\phi_j]]]\) for all \(i\neq j\).   Then \(\hat I^{\pm}\) become ideals of \(M\) and it holds \( \hat I=\hat I^+\oplus \hat I^-\).   
   We put \(\mathfrak{e}=E/\hat I^{+}\) and 
   \(\mathfrak{f}=F/\hat I^{-}\).      
On the other hand the subalgebra \(K\) of \(M\) is isomorphically mapped to a subalgebra \(\mathfrak{k}\) of \(\mathfrak{g}=M/\hat I\).   So  \(\mathfrak{k}\) is generated by  \(\mathbf{H}\otimes_{\mathbf{C}}\mathfrak{h}_o=\mathfrak{h}_o+J \mathfrak{h}_o\).    
Since \(\mathfrak{g}=M/\hat I\,\) and \(\, M=F\oplus K\oplus E\,\) it follows that 
\[\mathfrak{g}=\mathfrak{f}\oplus \mathfrak{k}\oplus \mathfrak{e}\,.\]
We continue to denote the generators of \(\mathfrak{g}\) by \(e_i,\,Je_i\,,\,h_i\,,\,Jh_i\,,\,f_i\,,\,Jf_i\,\).   These are the images of the generators of \(L\) under the natural homomorphism \(L\longrightarrow\,\mathfrak{g}\).
 
Let \(\mathfrak{h}_r=H_r/\hat I\).   \(H_r\) is isomorphically mapped to \(\mathfrak{h}_r\).   \(H_r\) being a maximal abelian subalgebra of \(M\), \(\mathfrak{h}_r\) is a maximal  abelian subalgebra of \(\mathfrak{g}\).      The root space decomposition of \(\mathfrak{g}\) with respect to \(\mathfrak{h}_r\)  is given by
\begin{eqnarray*}
&\mathfrak{g}\,=\,\sum_{\lambda\in\Phi} \mathfrak{g}_\lambda\,, \\[0.2cm]
& \mathfrak{h}_r\,\subset\,\mathfrak{g}_0\,,\quad  \mathfrak{g}_{\lambda}\,=\{\xi\in \mathfrak{g};\,ad(h)\,\xi\,=\lambda(h)\xi,\quad \forall h\in \mathfrak{h}_r\}  \quad\mbox{
for \(\lambda\in\Phi\)} .
\end{eqnarray*}
Since \(M_0=K\) we also have \(\mathfrak{g}_0=\mathfrak{k}\).    Now \(ad\,e_i\) and  \(ad\,f_i\) are nilpotent endomorphisms of \(\mathfrak{g}_o\),  so are the endomorphisms \(ad\,\epsilon_i\) and \(ad\,\phi_i\) nilpotent.   We know already that the root system \(\Phi\) and the Weyl group of transformations of \(H_r\) are finite.   Then the finiteness of \(\dim \mathfrak{g}\) follows by a similar discussion as in the proof of Theorem\ref{Serre}, \cite{H,C}.

\begin{theorem}
\begin{enumerate}
\item
\begin{eqnarray}
\mathfrak{g}&=&\mathfrak{g}_0\,+\,\mathfrak{e}\,+\,\mathfrak{f}\,,\\[0.2cm]
 \mathfrak{e}&=&\,\sum_{\lambda\in \Phi^+}\mathfrak{g}_\lambda\,,\quad  \mathfrak{f}\,=\,\sum_{\lambda\in \Phi^-}\mathfrak{g}_\lambda  \nonumber
\end{eqnarray}
where \(\Phi^{\pm}=\{n_1\alpha_1+\cdots+n_l\alpha_l\neq 0\,;\,n_i\geq 0\,\mbox{ ( resp. \(n_i\leq 0\) )  }\quad\forall i\,\,\} \).
\item
\begin{eqnarray*}
\mathfrak{g}_{\lambda}&=&\,\mathbf{H}\otimes (\mathfrak{g}_o)_{\lambda} \quad\mbox{ if }\,\lambda\neq 0 .\\[0.2cm]
\mathfrak{g}_0&=&\,\mathfrak{k}\,,\quad \mathfrak{e}\,=\,\mathbf{H}\otimes \mathfrak{e}_o\,,\quad \mathfrak{f}\,=\,\mathbf{H}\otimes \mathfrak{f}_o\,.
\end{eqnarray*}

\item
\begin{equation}
\dim_{\mathbf{C}}\,\mathfrak{g}_{\alpha_i}=2\,,  \quad \dim_{\mathbf{C}}\,\mathfrak{g}_{-\alpha_i}=2\,, \quad i=1,\cdots,l.  
\end{equation} 
\item
\begin{equation}
[\,\mathfrak{k}\,,\,\mathfrak{k}\,]\,=\,\mathfrak{h}^{\bot}_r\,,
\end{equation}
where \(\mathfrak{h}^{\bot}_r\) is the complementary ideal to \(\mathfrak{h}_r\) in \(\mathfrak{k}\):
\(\,\mathfrak{k}\,=\,\mathfrak{h}_r\,+\,\mathfrak{h}^{\bot}_r\).
\end{enumerate}
 \end{theorem}

\begin{proof}~~~
We have
\begin{eqnarray*}
ad(h_i)(e_j)&=&\alpha_j(h_i)e_j\,,\quad ad(h_i)(Je_j)=\alpha_j(h_i)Je_j\,, \\[0.2cm]
ad(h_i)(f_j)&=&-\alpha_j(h_i)f_j\,,\quad ad(h_i)(Jf_j)=-\alpha_j(h_i)Jf_j\,
\end{eqnarray*}
Hence \(\dim\,\mathfrak{g}_{\alpha_j}\geq 2\).   
As we have seen in steps 5 and 6 of the proof of Theorem \ref{quatofsimpleLie}, \(\,\mathfrak{e}\) is generated by \(\{e_i\,,\,Je_i\,;\,i=1,\,\cdots,l\,\}\) so is spanned by monomials in these elements.    
 All such monomials are weight vectors.    Since \(\alpha_j\,;\,1\leq j\leq l\), are linearly independent the only monomials which have weight \(\alpha_j\) are \(e_j\) and \(Je_j\).   
Hence \(\mathfrak{g}_{\alpha_j}=(\mathfrak{g}_o)_{\alpha_j}+J(\mathfrak{g}_o)_{\alpha_j}\).   Thus \(\dim\,\mathfrak{g}_{\alpha_j}=2\).   Similarly \(\dim \,\mathfrak{g}_{-\alpha_j}=2\).    It follows that \(\mathfrak{g}_{\lambda}=(\mathfrak{g}_o)_{\lambda}+J(\mathfrak{g}_o)_{\lambda}=\mathbf{H}\otimes_{\mathbf{C}}(\mathfrak{g}_o)_{\lambda}\) for any root \(\lambda\neq 0\).    
The others are proved in a routine method.   The last assertion follows from Lemma \ref{K} and (\ref{1bracket})
\end{proof}


{\bf Acknowledgement.}

I would like to acknowledge that Professor J. Sekiguchi of Tokyo university of agriculture and technology suggested me several fundamental facts on compact Lie algebras.     Professor K. Furutani of Tokyo university of science read the early version of this work and gave me several advices.      My thanks goes also to Professor Y. Yoshii of Iwate University for his interest to this work.

\end{document}